\documentclass[11pt,a4 paper,twoside, reqno]{amsart}
\usepackage[utf8]{inputenc}
\usepackage[OT1]{fontenc}
\usepackage[english]{babel}
\usepackage{dsfont}
\usepackage{amsfonts}
\usepackage{amsmath}
\usepackage{wasysym}
\usepackage{amsthm}
\usepackage{amssymb}
\usepackage{float}
\usepackage{textcomp}
\usepackage{xcolor}
\usepackage{graphicx}
\usepackage{fancybox}
\usepackage{fancyhdr}
\usepackage{mathrsfs}
\usepackage{tikz}
\usetikzlibrary{arrows}
\usetikzlibrary{calc}
\usepackage{centernot}
\usepackage{listings}
\usepackage{faktor}
\usepackage{bm}
\usepackage{setspace}
\usepackage{hyperref}
\usepackage{newlfont}
\usepackage{geometry}
\usepackage{ccaption}
\usepackage{tipa}
\usepackage{units}
\usepackage[autostyle,italian=guillemets]{csquotes}
\usepackage{comment}

\usepackage{bbm}
\usepackage{bbold}
\usepackage{guit}
\usepackage{tikz-cd}
\usepackage{enumitem}

\everymath{\displaystyle}

\theoremstyle{definition}
\newtheorem{Def}{Definition}[section]

\newtheorem{nota}[Def]{Notation}
\newtheorem{es}[Def]{Example}

\newtheorem{as}[Def]{Assumption}

\theoremstyle{remark}
\newtheorem{obs}[Def]{Remark}
\theoremstyle{plain}
\newtheorem{prop}[Def]{Proposition}

\newtheorem{lema}[Def]{Lemma}

\newtheorem{cor}[Def]{Corollary}

\newtheorem{teo}[Def]{Theorem}

\newtheorem*{teo:char}{Theorem~\ref{char-single}}
\newtheorem*{teo:char2}{Theorem~\ref{birk-theo}}

\newcommand{\bb}{\mathbbm}

\newcommand{\bo}{\mathbf}
\newcommand{\A}{{\mathcal A}}

\newcommand{\C}{{\mathcal C}}
\newcommand{\D}{{\mathcal D}}
\newcommand{\E}{{\mathcal E}}

\newcommand{\G}{{\mathcal G}}

\newcommand{\I}{{\mathcal I}}

\newcommand{\K}{{\mathcal K}}

\newcommand{\M}{{\mathcal M}}

\newcommand{\T}{{\mathcal T}}

\newcommand{\V}{{\mathcal V}}

\newcommand\Cat{\operatorname{\bo{Cat}}}
\newcommand\Set{\operatorname{\bo{Set}}}
\newcommand\Met{\operatorname{\bo{Met}}}

\newcommand\Pos{\operatorname{\bo{Pos}}}

\newcommand{\LL}{{\mathbbm L}}
\newcommand{\EE}{{\mathbbm E}}
\newcommand\Str{\operatorname{\textnormal Str}}
\newcommand\Mod{\operatorname{\textnormal Mod}}

\newcommand{\tx}{\textnormal}

\newcommand{\tsum}{\textstyle\sum}
\newcommand{\dw}{\downarrow\kern-3pt}

\newcommand{\op}{^\textnormal{op}}

\newcommand{\colim}{\operatornamewithlimits{colim}}

\makeatother

\makeatletter
\newcommand{\changeoperator}[1]{%
	\csletcs{#1@saved}{#1@}%
	\csdef{#1@}{\changed@operator{#1}}%
}
\newcommand{\changed@operator}[1]{%
	\mathop{%
		\mathchoice{\textstyle\csuse{#1@saved}}
		{\csuse{#1@saved}}
		{\csuse{#1@saved}}
		{\csuse{#1@saved}}%
	}%
}
\makeatother

\makeatletter
\def\@tocline#1#2#3#4#5#6#7{\relax
	\ifnum #1>\c@tocdepth % then omit
	\else
	\par \addpenalty\@secpenalty\addvspace{#2}%
	\begingroup \hyphenpenalty\@M
	\@ifempty{#4}{%
		\@tempdima\csname r@tocindent\number#1\endcsname\relax
	}{%
		\@tempdima#4\relax
	}%
	\parindent\z@ \leftskip#3\relax \advance\leftskip\@tempdima\relax
	\rightskip\@pnumwidth plus4em \parfillskip-\@pnumwidth
	#5\leavevmode\hskip-\@tempdima
	\ifcase #1
	\or\or \hskip 1em \or \hskip 2em \else \hskip 3em \fi%
	#6\nobreak\relax
	\hfill\hbox to\@pnumwidth{\@tocpagenum{#7}}\par% <---- \dotfill -> \hfill
	\nobreak
	\endgroup
	\fi}
\makeatother

\changeoperator{sum}
\changeoperator{prod}
\changeoperator{coprod}

\title[On enriched terms and 2-categorical universal algebra]{On enriched terms and 2-categorical\\ universal algebra}
\author{Giacomo Tendas}
\address{Department of Mathematics, University of Manchester, Faculty of Science and Engineering, Alan Turing Building, M13 9PL Manchester, UK}
\email{giacomo.tendas@manchester.ac.uk}
\date{\today}
\thanks{The author acknowledges with gratitude the support of the EPSRC postdoctoral fellowship EP/X027139/1}

\begin{document}
	
\begin{abstract}
 	We introduce a new notion of recursively generated {\em enriched term} which generalizes the one studied in joint work with Rosick\'y. These new terms come together with a notion of {\em term-interpretability}, which recovers the same type of interpretability that has been considered for enrichment over posets, metric spaces, and $\omega$-complete posets. As an application of this, we specialize to the 2-categorical case by considering 2-dimensional terms and 2-dimensional equational theories. In this context we also give an explicit description of free structures and prove a 2-dimensional Birkhoff variety theorem. 
\end{abstract}	
	
\maketitle
\begin{comment}
	{\small
		\noindent{\bf Keywords:} Flatness, regular/exact categories, free completions, lex colimits, enriched categories.\\
		{\bf Mathematics Subject Classification:} 18E08, 18A35, 18D20, 18B25.\\
		{\bf Competing Interests:} the author declares none.
	}
\end{comment}

\setcounter{tocdepth}{1}
\tableofcontents

\section{Introduction}

Several enriched approaches to Universal Algebra have been presented by different authors in recent times. Both from a purely categorical point of view, which is a direct enriched analogue of Lawvere theories~\cite{Law73:articolo} (see~\cite{power1999enriched,BG,LP23}), and from a more logical one, which involves enriched notions of terms and equations (see~\cite{fiore2008term,LP23dg}). With respect to this latter approach the theory was still incomplete, with several versions of universal algebra being developed for specific bases of enrichment, and the notion of term appearing in the literature not being recursively generated as in the ordinary setting. The recent paper~\cite{RT23EUA}, in joint work with J. Rosick\'y, aimed to unify the framework and to provide a more tractable notion of term.

In fact, in~\cite{RT23EUA} we introduced enriched notions of languages and recursively generated enriched terms which allowed to present $\V$-categories of models of $\lambda$-ary enriched monads as $\V$-categories of models of equational theories defined by term equalities. This framework was later used in~\cite{RT25} to define atomic formulas and a regular fragment of enriched logic. 

Every enriched term $\tau$, as every function symbol of an enriched language, comes together with {\em input} and {\em output} arities which are objects of $\V$: we write $\tau:(X,Y)$ to mean that $\tau$ has input arity $X$ and output arity $Y$. On an $\LL$-structure $A$ such a term $\tau$ is interpreted as a morphism
$$ \tau_A\colon A^X\longrightarrow A^Y $$
in $\V$, where $A^{(-)}:=[-,A]$ denotes the internal hom.

%In~\cite[Theorem~5.23]{RT23EUA} we proved that, given an enriched language $\LL$, to present the $\V$-categories of models of equational theories on $\LL$, it is enough to consider terms whose output arities vary on a generator $\G$ of $\V_0$. However, in that framework such simplification could not be carried out on the function symbols of the language $\LL$ itself. 

While passing from (input) arities being natural numbers (as in ordinary universal algebra) to them being objects of $\V$ might seem natural, the fact that each function symbol (or term) come also with an output arity is somewhat dissatisfying. Nonetheless, this assumption is needed since in general the unit $I$ does not generate $\V_0$ under colimits (like in $\Set$). One of our purposes is to modify the notions of~\cite{RT23EUA} so that function symbols and terms have output arities in a generator $\G\subseteq\V_0$. This will allow us to deal with much simpler terms, and to recover several works in the literature within our framework.

More in detail, fixed a regular cardinal $\lambda$, a generator $\G$ of $\V_0$, and a suitable set $\Gamma$ of ``generating $\lambda$-ary epimorphisms'' (Assumption~\ref{assu}) we introduce the notion of $\Gamma$-ary language and term so that:\begin{enumerate}[leftmargin=1cm]
	\item[(a)] the output arities of both $\Gamma$-ary function symbols and terms are allowed to vary among $\G$;
	\item[(b)] the substitution rule allows to ``glue'' $\Gamma$-ary terms together, subject to a certain covering property specified by $\Gamma$.
\end{enumerate}
What we mean exactly by (a) and (b) above will be explained in detail in Section~\ref{structures-and-terms}. 

As mentioned above, the goal of point (a) is to make our enriched languages and theories more elementary. For instance, when the unit $I$ of $\V$ is a generator this means that we can consider the output arities in our languages and equations to be trivial like in the ordinary $\bo{Set}$-enriched case. This allows us to recover exactly the same notion of terms considered interdependently for enrichment over posets~\cite{adamek2021finitary}, metric spaces~\cite{mardare2016quantitative}, and $\omega$-complete posets~\cite{adamek1985birkhoff}. Moreover, for the case of $\V=\Cat$, which we study in detail, condition (a) implies that it will be enough to consider 1-dimensional (with output arity $1$) and 2-dimensional (with output arity $\bb 2$) function symbols and terms.

Point (b), instead, makes this notion of term more ``malleable'' than the one considered in~\cite{RT23EUA} and in particular will lead to the introduction of {\em term interpretability}, meaning that a given term may not be interpretable in {\em all} structures. To better understand what this means, let us give an example in the 2-categorical setting:

\begin{es}
	 Fix $\V=\Cat$; here we can consider two 2-dimensional terms $\sigma$ and $\tau$ of (input) arity $X\in\Cat_f$ (and output arity $\bb 2$) over a language $\LL$. On an $\LL$-structure $A$, these are interpreted as maps $\sigma_A,\tau_A\colon A^X\to A^\bb 2$, or equivalently as natural transformations
	$$ \sigma_A\colon(\sigma_0)_A\Rightarrow(\sigma_1)_A\colon A^X\to A\ \ \ \ \ \  \tx{ and }\ \ \ \ \ \ \tau_A\colon(\tau_0)_A\Rightarrow(\tau_1)_A\colon A^X\to A. $$
	The new substitution rule of point (b) allows us to form a new 2-dimensional $X$-ary term $\sigma\circ\tau$, which is not guaranteed to be interpretable over all $\LL$-structures. In fact, $\sigma\circ\tau$ is interpretable over $A$ if and only if $(\tau_1)_A=(\sigma_0)_A$ and in that case its interpretation is given by the composite $\sigma_A\circ\tau_A$ of the two natural transformations. 
\end{es}

For the bases of enrichment $\Pos$, $\Met$, and $\omega$-$\bo{CPO}$, a notion of interpretability has been considered in the literature before (see~\cite{adamek2021finitary,mardare2016quantitative,adamek1985birkhoff}); we prove that it coincides with the one that we introduce in Examples~\ref{Pos},\ref{Met}, and~\ref{CPO}.

Given this new notion of enriched term, we define $\Gamma$-ary equational theories and prove that these are at least as expressive as those introduced in~\cite{RT23EUA}, and we characterize the $\V$-categories of models of $\Gamma$-ary equational theories as those arising as $\V$-categories of algebras of a $\lambda$-ary monad on $\V$:

\begin{teo:char}
	The following are equivalent for a $\V$-category $\K$: \begin{enumerate}[leftmargin=1cm]
		\item $\K\simeq\Mod(\EE)$ for a $\Gamma$-ary equational theory $\EE$;
		\item $\K\simeq\tx{Alg}(T)$ for a $\lambda$-ary monad $T$ on $\V$;
		\item $\K$ is cocomplete and has a $\lambda$-presentable and $\V$-projective strong generator $G\in\K$;
		\item $\K\simeq\lambda\tx{-Pw}(\T,\V)$ is equivalent to the $\V$-category of $\V$-functors preserving $\lambda$-small powers, for some $\V_\lambda$-theory $\T$.
	\end{enumerate}
\end{teo:char}

In Section~\ref{souns-section} we generalize this to the case where $\lambda$-filtered colimits are replaced by $\Phi$-flat colimits, built from a sound class $\Phi$ in the sense of~\cite{ABLR02:articolo}. In particular we characterize $\V$-categories of algebras of strongly finitary monads over a cartesian closed base ($\Phi$ being the class for finite products).

We conclude the paper by treating the 2-categorical case (Section~\ref{2-cat-sect}); for simplicity we restrict to the finitary setting. In Definition~\ref{cat-term} we spell out how 2-categorical terms are defined, and then outline a few term-constructions to show how the more general substitution rule works (Remark~\ref{sample-terms}). As an example of equational theory, we give equations that present the 2-category of categories with chosen (co)limits of a given shape (Example~\ref{limits}).

As stated before, given a language $\LL$ and a finitely presentable arity $X\in\Cat_f$ we will have a notion of 1-dimensional and 2-dimensional $X$-ary terms. These, as we shall see in Theorem~\ref{free-structures-teo}, describe (up to a certain equivalence relation) the objects and morphisms of the free $\LL$-structure $FX\in\Str(\LL)$ on $X$; that is, the value at $X$ of the left adjoint to the forgetful functor
$$ U\colon\Str(\LL)\longrightarrow\Cat. $$
Such an explicit construction of free structures will allow us to prove a general Birkhoff variety theorem for 2-dimensional universal algebra:
\begin{teo:char2}
	Let $\V=\Cat$ and $\LL$ be a finitary language. The following are equivalent for a full subcategory $\A$ of $\Str(\LL)$:\begin{enumerate}[leftmargin=1cm]
		\item $\A\cong\Mod(\EE)$, for some finitary equational theory $\EE$ on $\LL$;
		\item $\A$ is closed un $\Str(\LL)$ under products, powers, strong subobjects, $\V$-split quotients, and filtered colimits.
	\end{enumerate}
\end{teo:char2}
\noindent It remains open whether this and a similar description of free structures can be obtained in a more general enriched setting (which includes the one above); this would improve the results of~\cite[Section~6]{RT23EUA}.

\section{Background}\label{back}

We fix as base of enrichment a symmetric monoidal closed category $\V=(\V_0,\otimes,I)$ which is complete and cocomplete, and we follow the notation of Kelly~\cite{Kel82:libro} for matters about enrichment over $\V$. The internal hom of $\V$ is usually denoted by $[-,-]$; however, when talking about $\LL$-structures we will denote the internal hom as follows
$$A^X:=[X,A];$$
this is to give a more intuitive interpretation of arities and function symbols.

We assume $\V$ to be locally $\lambda$-presentable as a closed category~\cite{Kel82:articolo}, for some fixed regular cardinal $\lambda$. This means that $\V_0$ is locally $\lambda$-presentable and the full subcategory $(\V_0)_\lambda$ spanned by the $\lambda$-presentable objects is closed under the monoidal structure of $\V_0$. Following \cite{Kel82:articolo}, we say that a $\V$-category $\K$ is {\em locally $\lambda$-presentable} if it is cocomplete (all conical colimits and copowers exist) and has a strong generator $\G$ made of $\lambda$-presentable objects (that is, $\K(G,-)\colon \K\to\V$ preserves $\lambda$-filtered colimits for any $G\in\G$, and they jointly reflect isomorphisms). 

Next we recall the notions of language and structure introduced in \cite{RT23EUA}.

\begin{Def}[\cite{RT23EUA}]\label{language}
	A (single-sorted) {\em language} $\LL$ over $\V$ is the data of a set of function symbols $f\colon(X,Y)$ whose arities $X$ and $Y$ are objects of $\V$. The language is called $\lambda$-ary if the input and output arities of all function symbols lie in $\V_\lambda$.
\end{Def}

\begin{Def}[\cite{RT23EUA}]
	Given a language $\LL$, an {\em $\LL$-structure} is the data of an object $A\in\V$ together with a morphism $$f_A\colon A^X\to A^Y$$ in $\V$ for any function symbol $f\colon(X,Y)$ in $\LL$.
	
	A {\em morphism of $\LL$-structures} $h\colon A\to B$ is the data of a map $h\colon A\to B$ in $\V$ making the following square commute
	\begin{center}
		\begin{tikzpicture}[baseline=(current  bounding  box.south), scale=2]
			
			\node (a0) at (0,0.8) {$A^X$};
			\node (b0) at (1,0.8) {$B^X$};
			\node (c0) at (0,0) {$A^Y$};
			\node (d0) at (1,0) {$B^Y$};
			
			\path[font=\scriptsize]
			
			(a0) edge [->] node [above] {$h^X$} (b0)
			(a0) edge [->] node [left] {$f_A$} (c0)
			(b0) edge [->] node [right] {$f_B$} (d0)
			(c0) edge [->] node [below] {$h^Y$} (d0);
		\end{tikzpicture}	
	\end{center} 
	for any $f\colon(X,Y)$ in $\LL$.
\end{Def}

The $\V$-category $\Str(\LL)$ of $\LL$-structured is defined as in \cite[Definition~3.3]{RT23EUA}; this has $\LL$-structures as objects, morphisms of $\LL$-structures as arrows, and comes together with a forgetful $\V$-functor $U\colon \Str(\LL)\to \V$, sending an $\LL$-structure to its underlying object, which has a left adjoint $F$. Note that (by~\cite[Lemma~5.5]{tendas2024more}) these three properties univocally determine the $\V$-category $\Str(\LL)$.

\section{Yet another notion of term}\label{structures-and-terms}

In this section we introduce the new notion of term which generalizes that of~\cite{RT23EUA}; this will involve the choice of a generator of $\V_0$ and of a certain set of epimorphisms $\Gamma$.

\begin{as}\label{assu}
	Given $\V$ as in the Section~\ref{back} above, we fix a generator $\G$ of $\V_0$ consisting of $\lambda$-presentable objects. In addition we fix a class $\Gamma$ of maps in $\V$ of the form 
	$$ e\colon\sum_{j\in J} G_j\twoheadrightarrow Y$$
	where:\begin{enumerate}[leftmargin=1cm]
		\item every such map is an epimorphism, $J$ is $\lambda$-small, and every $G_j$ lies in $\G$;
		\item every $Y\in\V_\lambda$ is the codomain of some map in $\Gamma$;
		\item for any $G\in\G$ the identity $1_G\colon G\to G$ lies in $\Gamma$;
		\item $\Gamma$ is closed in $\V^\bb 2$ under $\lambda$-small coproducts;
		\item given $G\in\G$, $ e\colon\sum_{j\in J} G_j\twoheadrightarrow Y$ and $e_{j}\colon \sum_{i\in I_j} G_{ij}\twoheadrightarrow G\otimes G_j$ in $\Gamma$, for any $j\in J$, then also the composite
		$$ \sum_{i\in I_j,j\in J} G_{ij}\xrightarrow{\ \sum_j e_j\ } \sum_{j\in J}G\otimes G_j\cong G\otimes\sum_{j\in J} G_j \xrightarrow{\ G\otimes e\ } G\otimes Y $$
		is in $\Gamma$.
	\end{enumerate}
\end{as}

More interesting examples will be given later, for now we just note that given a generator $\G$ there are always suitable choices of $\Gamma$:

\begin{es}
	We can always take $\Gamma$ to consist of {\em all} the epimorphisms of the requested domain and codomain. If $\G$ is a strong generator, then we can take $\Gamma$ to be all the strong epimorphisms of the relevant form. If $\G=\V_\lambda$ we can take $\Gamma$ to be just the isomorphisms between $\lambda$-presentable objects.
\end{es}

The notion of language relevant in this context is the following:

\begin{Def}
	A $\Gamma$-ary language $\LL$ over $\V$ is a $\lambda$-ary language, in the sense of Definition~\ref{language}, whose function symbols have output arities in $\G$.
\end{Def}

Over such languages, we can now introduce the notion of $\Gamma$-ary term:

\begin{Def}\label{Sigma-terms}
	Let $\LL$ be a $\Gamma$-ary language over $\V$. The class of \textit{$\Gamma$-ary $\LL$-terms} is defined recursively as follows:
	\begin{enumerate}[leftmargin=1cm]
		\item Every morphism $f\colon G\to X$, with $G\in\G$ and $X\in \V_\lambda$, is an $(X,G)$-ary term.
		\item Every function symbol $f:(X,G)$ of $\LL$ is an $(X,G)$-ary term.
		\item If $t$ is a $(X,G)$-ary term and $P$ is in $\G$, then $t^P_h$ is a $(P\otimes X,Q)$-ary term for any $h\colon Q\to P\otimes G$ appearing as one component of some $e\colon \sum Q_j\twoheadrightarrow P\otimes G$ in $\Gamma$.
		\item Consider $e\colon \textstyle\sum_{j\in J} G_j\twoheadrightarrow Y$ in $\Gamma$. Given a family  $t_J=(t_j)_{j\in J}$, where $t_j$ is an $(X,G_j)$-ary term, and $s$ a $(Y,G)$-ary term; then $s(t_{J,e})$ is a $(X,G)$-ary term.
	\end{enumerate}

\end{Def}

We interpret $\Gamma$-ary terms on an $\LL$-structure as follows:

\begin{Def}
	Let $A$ be an $\LL$-structure, then the {\em interpretability} and {\em interpretation} of $\Gamma $-ary terms over an $\LL$-structure $A$ is defined recursively as follows:
	\begin{enumerate}[leftmargin=1cm]
		\item Every morphism $f\colon G\to X$, with $G\in\G$ and $X\in \V_\lambda$, is interpretable over $A$ and its interpretation is given by $$f_A:= A^f\colon A^X\to A^G;$$
		\item Every function symbol $f:(X,G)$ of $\LL$ is interpretable over $A$ and its interpretation is given by the map $$f_A\colon A^X\to A^G$$ which is part of the $\LL$-structure on $A$;
		\item If $t$ is an $(X,G)$-ary term that is interpretable over $A$, and $P\in\G$, then $t^P_h$ is interpretable over $A$ for any $h\colon Q\to P\otimes G$ as specified, and its interpretation is given by the map 
		$$ (t^Z_h)_A\colon A^{P\otimes X}\xrightarrow{\ (t_A)^Z \ } A^{P\otimes Y}\xrightarrow{A^{h}}A^{Q}.$$
		\item Given $e\colon \tsum_{j\in J} G_j\twoheadrightarrow Y$ in $\Gamma$, if $t_{J,e}=(t_j)_{j\in J}$ is formed by $(X,G_j)$-ary terms, and $s$ is a $(Y,G)$-ary term, then 
		$s(t_{J,e})$ is interpretable over $A$ if and only if $s$ and each $t_j$ are interpretable and there exists a (necessarily unique) map $(t_{J,e})_A$ as below.
		\begin{center}
			\begin{tikzpicture}[baseline=(current  bounding  box.south), scale=2]
				
				\node (a0) at (0,0.8) {$A^X$};
				\node (b0) at (1.2,0.8) {$A^Y$};
				\node (c0) at (1.2,0) {$A^{\sum_jG_j}$};
				
				\path[font=\scriptsize]
				
				(a0) edge [dashed, ->] node [above] {$(t_{J,e})_A$} (b0)
				(a0) edge [->] node [below] {$(t_j)_A\ \ \ \ \ \ $} (c0)
				(b0) edge [>->] node [right] {$A^e$} (c0);
			\end{tikzpicture}	
		\end{center}
		In that case the interpretation of $s(t_{J,e})$ is given by the composite
		$$ s(t_{J,e})_A\colon A^X\xrightarrow{\ (t_{J,e})_A  \ } A^Y\xrightarrow{\ s_A\ } A^G.$$
	\end{enumerate}
\end{Def}

\begin{Def}\label{equivalent}
	We say that two $(X,G)$-ary terms $s,t$ are {\em equivalent}, and write $s\equiv t$, if for any $\LL$-structure $A$, $s$ is interpretable over $A$ if and only if $t$ is, and in that case their interpretation coincide.
\end{Def}

\begin{obs}
	Is there a system of deduction rules that characterizes when two terms are equivalent? If so the definition of equivalence above would be purely syntactic.
\end{obs}

\begin{obs}\label{lessrules}
	If in rule (3) the object $P\otimes G$ is still in $\G$, then it is enough to apply the rule only to $h=1_{P\otimes G}$. Indeed all the remaining terms constructed using $h\colon Q\to P\otimes G$ can be obtained in an equivalent way by applying rule (4) with $e=1_{P\otimes G}$, to $s=h$ (given by rule (1)) and $t=t^Z_{1_{P\otimes G}}$.
\end{obs}

\begin{es}\label{set-case}
	By taking $\V=\Set$ with $\lambda=\aleph_0$, $\G=\{1\}$, and $\Gamma$ to consist of the isomorphisms, we recover the classical definition of term. Rule (1) declares variables, and says that for any finite set of variables $(x_1,\dots,x_n)$ we have $n$-ary terms that pick out one of the variables; giving $n$-ary projections. Rule (2) is standard and (3) is trivial (since $h$ is the identity on $1$). Rule (4) can be understood as the usual superposition of terms. 
	All terms are interpretable since the compatibility condition in rule (4) is always satisfied.
\end{es}

\begin{es}\label{iso-sigma}
	Consider a general $\V$, with $\G=\V_\lambda$ and $\Gamma$ given by the isomorphisms. Then, up to some redundancies generated by rule (3) (see Remark~\ref{lessrules}), we recover the notion of term from \cite{RT23EUA}. Also in this case all terms are interpretable.\\
	It is perhaps worth mentioning that interpretability, however, will not make these new $\Gamma$-ary terms suitable any more to construct other fragments of logic, for which one would like the interpretation of a given formula to always be defined. For that kind of framework the terms of~\cite{RT23EUA} remain a more suitable choice. 
\end{es}

Note that these $\Gamma$-ary terms may not be extended terms in the sense of \cite[Definition~4.4]{RT23EUA}; that is, they may not correspond to morphisms $FG\to FX$ in $\Str(\LL)$. This is because the interpretation of a $\Gamma$-ary term may not be defined for every $\LL$-structure $A$, while that of an extended term always is.

\begin{es}\label{concrete-case}
	Assume that the unit $I$ of $\V_0$ is a generator, so that $$U:=\V_0(I,-)\colon\V_0\to\Set$$ is conservative. Assume moreover that $U$ restricts to $U_\lambda\colon \V_\lambda\to \Set_\lambda$; then we can choose $\G=\{I\}$ and $\Gamma$ to consist of those maps $e\colon \textstyle\sum_S I\to X$ for which $Uf$ is bijective (these are epimorphisms since $I$ is a generator). \\
	In this context, function symbols and terms have only input-arities (that is, the output arities are trivial), rule (3) in the definition of $\Gamma$-term becomes redundant (Remark~\ref{lessrules}), and rule (4) is the usual superposition. Moreover, since for any $X\in\V$, to give a map $I\to X$ in $\V$ is the same as to give an element $x\colon 1\to UX$, it follows that in rule (1) a map $f\colon I\to X$ is simply a ``variable''  $x\in UX$. 
	
	Now, given a $\Gamma$-ary language $\LL$, let $\LL_0$ be the ordinary language which has the same function symbols as $\LL$ but whose arities are obtained by applying $U$: if $f$ is $X$-ary in $\LL$ then it is $UX$-ary in $\LL_0$, where we know that $UX$ is $\lambda$-small by hypothesis.
	
	The arguments above show that a $\Gamma$-ary $\LL$-term is the same as an ordinary term over $\LL_0$ (Example~\ref{set-case}). Nonetheless, while all $\LL_0$-terms are interpretable in a $\LL_0$-structure, not all $\Gamma$-ary $\LL$-terms are interpretable. In fact, when taking $\V$ to be any of the monoidal closed categories $\Pos$ of posets, $\omega$-$\bo{CPO}$ of posets with joints of $\omega$-chains, and $\Met$ of metric spaces, we recover exactly the notion of interpretability of~\cite{adamek2021finitary, adamek1985birkhoff, mardare2016quantitative}.
\end{es}

In the next two propositions we show how to construct other kinds of terms from our four rules. These will be useful in Section~\ref{equational-theories} below.

\begin{prop}\label{precomp-term}
	 If $s$ is a $(X,G)$-ary term and $h\colon X\to Y$ is a morphism in $\V_\lambda$, then there is a $(Y,G)$-ary term $s(h)$ which is interpretable in an $\LL$-structure $A$ if and only if $s$ is interpretable in $A$, and in that case $s(h)_A$ is given by the composite
	 $$s(h)_A \colon A^{Y}\xrightarrow{\ A^h \ } A^X\xrightarrow{\ s_A\ } A^G.$$
\end{prop}
\begin{proof}
	Consider a morphism $ e=(e_j)_{j\in J}\colon\textstyle\sum_{j\in J} G_j\twoheadrightarrow Y$ in $\Gamma$, then for any $j\in J$ the maps $h_j:=h\circ e_j$ define $(Y,G_j)$-ary terms (by rule (1)). Set then
	$$ s(h):= s(h_{J,e}); $$
	following rule (4). It is easy to see that this satisfies the required property.
\end{proof}

The following is an alternative approach to rule (4) which is closer to usual set-theoretic superposition of terms which stacks together all the input arities, as well as to the last rule of~\cite[Definition~4.1]{RT23EUA}.

\begin{prop}\label{sum-domain}
	For any family of terms $(t_j)_{j\in J}$, where $t_j$ is $(X_j,G_j)$-ary, any $ e=(e_j)_{j\in J}\colon\textstyle\sum_{j\in J} G_j\twoheadrightarrow Y$ in $\Gamma$, and any $(Y,G)$-ary term $s$ there is a $(\tsum_{j\in J}X_j,G)$-ary term $s(\hat t_J)$ whose interpretation on an $\LL$-structure $A$ is defined if and only if all terms $s$ and $t_j$ are interpretable and there exists a (necessarily unique) map $(\hat t_{J,e})_A$ as below.
	\begin{center}
		\begin{tikzpicture}[baseline=(current  bounding  box.south), scale=2]
			
			\node (a0) at (0,0.8) {$A^{\sum_jX_j}$};
			\node (b0) at (1.2,0.8) {$A^Y$};
			\node (c0) at (1.2,0) {$A^{\sum_jG_j}$};
			
			\path[font=\scriptsize]
			
			(a0) edge [dashed, ->] node [above] {$(\hat t_{J,e})_A$} (b0)
			(a0) edge [->] node [below] {$\prod_j (t_j)_A\ \ \ \ \ \ $} (c0)
			(b0) edge [>->] node [right] {$A^e$} (c0);
		\end{tikzpicture}	
	\end{center}
	and in that case is given by the composite
	$$ s(\hat t_{J,e})_A\colon A^{\sum_jX_j}\xrightarrow{\ (\hat t_{J,e})_A \ } A^Y\xrightarrow{\ s_A\ } A^G.$$
\end{prop}
\begin{proof}
	Let $X:=\tsum_{j\in J}X_i$ and denote by $h_j\colon X_j\to X$ the coproduct inclusions; then by Proposition~\ref{precomp-term} above we have $(X,G_j)$-ary terms $t'_j:=t_j(h_j)$. It is then enough to define
	$$ s(\hat t_{J,e}):=s(t'_{J,e}) $$
	which by construction satisfies the required property.
\end{proof}

\section{$\Gamma$-ary equational theories}\label{equational-theories}

We now have all the ingredients needed to introduce the notion of $\Gamma$-ary equational theory:

\begin{Def}
	Given a $\Gamma$-ary language $\LL$, a $\Gamma$-ary equational theory $\EE$ is a set of judgements of the form:\begin{itemize}[leftmargin=0.8cm]
		\item $\dw t$, where $t$ is a $\Gamma$-ary term;
		\item $(s=t)$, where $s$ and $t$ are $\Gamma$-ary terms of the same arity. 
	\end{itemize}
	An $\LL$-structure $A$ is a model of $\EE$ if:\begin{itemize}[leftmargin=0.8cm]
		\item $t$ is interpretable over $A$, whenever $\dw t\in\EE$;
		\item $s$ and $t$ are interpretable over $A$ and $s_A=t_A$, whenever $(s=t)\in\EE$.
	\end{itemize} 
	We denote by $\Mod(\EE)$ the full subcategory of $\Str(\LL)$ spanned by the models of $\EE$. 
\end{Def}

\begin{nota}
	Given $e\colon\textstyle \sum_{j\in J} G_j\twoheadrightarrow Y$ in $\Gamma$ and terms $t_j:(X,G_j)$ we write
	$$ \dw t_{J,e} $$
	to mean $\dw e_j(t_{J,e})$, for some (or equivalently, any) $j\in J$.
\end{nota}

We see now some examples; the case of $\V=\Cat$ will be treated separately in Section~\ref{2-cat-sect}.

\begin{es}\label{Pos}
	Let $\V=\Pos$ with $\Gamma$ as in Example~\ref{concrete-case}. Given a $\Gamma$-ary language $\LL$ and $X\in\Pos_f$ of cardinality $n\in\mathbbm N$, an $X$-ary term over $\LL$ is the same as a $n$-ary term over $\LL_0$. Given two $n$-ary terms $t$ and $s$, then we can form the judgement 
	$\dw (t,s)_e$
	where $e\colon 1+1\to \bb 2$ the inclusion into $\bb 2=\{0\leq 1\}$. It follows by definition that an $\LL$-structure $A$ satisfies such judgement if and only if $$t_A(\bar a)\leq s_A(\bar a)$$ for any $\bar a\in A^n$. Therefore we recover the inequalities of \cite{adamek2021finitary}; since every judgement can be reduced to a set of inequalities as above, our $\Gamma$-ary equational theories are the same as their theories with inequalities.
\end{es}

\begin{es}\label{Met}
	Similarly, let $\V=\Met$ and $\LL$ be a $\Gamma$-ary language with $\Gamma$ as in Example~\ref{concrete-case}. Given $X\in\Met$ of cardinality $n\in\mathbbm N$, an $X$-ary term over $\LL$ is the same as a $n$-ary term over $\LL_0$. Given two $n$-ary terms $t$ and $s$ and some $\epsilon\in(0,\infty)$, then we can form the judgement 
	$\dw (t,s)_e$
	where $e\colon 1+1\to \bb 2_\epsilon$ the inclusion into the metric space having two points of distance $\epsilon$. It follows by definition that an $\LL$-structure $A$ satisfies such judgement if and only if $$d_A(t_A(\bar a), s_A(\bar a))\leq\epsilon$$ for any $\bar a\in A^n$. Therefore we recover the quantitative equations of \cite{mardare2016quantitative}. As in the previous example, every judgement can be reduced to a set of $\epsilon$-inequalities as above, making our $\Gamma$-ary equational theories the same as their quantitative equational theories.
\end{es}

\begin{es}\label{CPO}
	Let $\V=\omega$-$\bo{CPO}$ be the cartesian closed category of posets with joints of $\omega$-chains. We can still take $\G=\{1\}$, $\lambda=\aleph_1$, and consider $\Gamma$ to consist of the maps 
	$$ \sum_{j\in J}1 \to Y $$
	which are injective and dense (that is, the closure of the image under joins of $\omega$-chains, is the whole $Y$). Since $\G$ is a singleton we refer to $(X,1)$-ary terms simply as $X$-ary terms.\\ 
	In the term formation rule (4), we can take $Y:=\omega+1=\omega\cup\{\omega\}$ (where we are seeing $\omega$ as the chain of natural numbers, and $\omega\in\omega+1$ as a top element) and $e\colon \textstyle\sum_{n\in\omega}1\to \omega+1$ that picks the natural numbers in $\omega+1$. Let $\iota_\omega$ be the $(\omega+1)$-ary term corresponding to the inclusion $1\to\omega+1$ picking $\omega$. Given a family $(t_n)_{n\in\omega}$ of $X$-ary terms then we can form the new $X$-ary term
	$$ \bigvee_{n\in\omega}t_n:= \iota_\omega(t_{\omega,e}). $$
	It is easy to see that an $\LL$-structure $A$ satisfies the judgement
	$$ \dw  \bigvee_{n\in\omega}t_n$$
	if and only if for any $a\in A$ we have $(t_n)_A(a)\leq (t_{n+1})_A(a)$, for any $n\in\omega$, and in that case
	$$ (\bigvee_{n\in\omega}t_n)_A(a)=\bigvee_{n\in\omega}(t_n)_A(a) $$
	This coincides with the kind of terms and the interpretation considered in~\cite{adamek1985birkhoff}.
\end{es}

We can now move to proving the main properties of the $\V$-categories of models of $\Gamma$-ary equational theories. We begin with a technical lemma.

\begin{lema}\label{limits of terms}
	Given a weight $N\colon\C\to\V$ and a diagram $H\colon\C\to\Str(\LL)$. For any $\Gamma$-term $\tau:(X,G)$ over $\LL$, if $\tau$ is interpretable over $Hc$ for every $c\in \C$, then $\tau$ is interpretable over $\{N,H\}\in\Str(\LL)$. Moreover, in that case the following diagram 
	\begin{center}
		\begin{tikzpicture}[baseline=(current  bounding  box.south), scale=2]
			
			\node (a0) at (-0.5,0.8) {$\{N,H\}^X$};
			\node (b0) at (1.5,0.8) {$\{N,H\}^G$};
			\node (c0) at (-0.5,0) {$\{N,H(-)^X\}$};
			\node (d0) at (1.5,0) {$\{N,H(-)^G\}$};
			
			\path[font=\scriptsize]
			
			(a0) edge [->] node [above] {$\tau_{\{N,H\}}$} (b0)
			(c0) edge [->] node [below] {$\{N,\tau_{H(-)}\}$} (d0)
			(a0) edge [->] node [above] {} (c0)
			(b0) edge [->] node [above] {} (d0);
		\end{tikzpicture}	
	\end{center}
	commutes, where the vertical maps are the comparison isomorphisms.
	The same assertions hold if we replace weighted limits by $\lambda$-filtered colimits.
\end{lema}
\begin{proof}
	We prove this by induction on the definition of $\tau$. If $\tau$ is as in rules (1) and (2), then it is interpretable in any $\LL$-structure. If $\tau=t^Z_h$ as in rule (3), where $t$ satisfies the inductive hypothesis; then also $\tau$ is interpretable over $\{N,H\}$ since its interpretation is obtained by taking the composite of a power of $t_{\{N,H\}}$ with $\{N,H\}^h$.
	
	It remains to consider rule (4). Suppose that $\tau=s(t_{J,e})$ and that $s$ and $t_j$, for $j\in J$, satisfy the inductive hypothesis. We need to prove that there exists a dashed arrow as below.
	\begin{center}
		\begin{tikzpicture}[baseline=(current  bounding  box.south), scale=2]
			
			\node (a0) at (-0.5,0.8) {$\{N,H\}^X$};
			\node (b0) at (1.5,0.8) {$\{N,H\}^Y$};
			\node (c0) at (1.5,0) {$\{N,H\}^{\sum_jG_j}$};
			
			\path[font=\scriptsize]
			
			(a0) edge [dashed, ->] node [above] {$(t_{J,e})_{\{N,H\}}$} (b0)
			(a0) edge [->] node [below] {$(t_j)_{\{N,H\}}\ \ \ \ \ \ $} (c0)
			(b0) edge [>->] node [right] {$\{N,H\}^e$} (c0);
		\end{tikzpicture}	
	\end{center}
 	But 
 	$$  \{N,H\}^X\cong \{N,\ H(-)^X\}$$
 	and for each $C\in\C$ we have the map 
 	$$ (t_{J,e})_{H(C)}\colon H(C)^X\to H(C)^Y $$
 	defining the interpretability of $t_{J,e}$ over $H(C)$. By uniqueness this assignment extends to a $\V$-functor $(t_{J,e})_{H(-)}\colon \C\to \V^\bb 2$. Then it is enough to define $(t_{J,e})_{\{N,H\}}:= \{N,(t_{J,e})_{H(-)}\}$.
 	The fact that $ \tau_{\{N,H\}}\cong \{N,\tau_{H(-)}\}$ now follows easily. 
 	
 	The argument for $\lambda$-filtered colimits is exactly the same: given a $\lambda$-filtered diagram $H\colon \C\to\Str(\LL)$ for which $\tau$ is interpretable over $H(C)$ for any $C\in\C$, in the situation of rule (4) we define $(t_{J,e})_{\colim H}$ as the composite
 	$$ ( \colim H)^X\xrightarrow{\ \cong\ }\colim_{C\in\C} (H(C)^X)\xrightarrow{\colim (t_{J,e})_{H(-)}} \colim_{C\in\C} (H(C)^Y)\xrightarrow{\ \rho\ } (\colim H)^Y $$
 	where we only used that powers by $X$ commute in $\V$ with such colimits ($\rho$ above is the comparison map).
\end{proof}

Next we consider those kind of subobjects and quotients that which we prove shall preserve the validity of equations and interpretations:

\begin{Def}
	We say that a monomorphism $m\colon A\rightarrowtail B$ in $\Str(\LL)$ is {\em $\Gamma$-strong} if as a morphism of $\V$ it is right orthogonal to every map of the form $e\otimes Z$ for $e\in \Gamma$ and $Z\in\V$.\\
	A morphism $e\colon B\twoheadrightarrow C$ in $\Str(\LL)$ is called $\V$-split if it is a split epimorphism in $\V$. 	
\end{Def}

\begin{prop}\label{closure}
	Let $\EE$ be a $\Gamma$-ary equational theory; then $\Mod(\EE)$ is closed in $\Str(\LL)$ under:\begin{enumerate}[leftmargin=1cm]
		\item small limits.
		\item $\lambda$-filtered colimits.
		\item $\Gamma$-strong subobjects: if $m\colon A\rightarrowtail B$ is a $\Gamma$-strong monomorphism and $B\in\Mod(\EE)$ then also $A\in\Mod(\EE)$.
		\item $\V$-split quotients: if $e\colon B\twoheadrightarrow C$ is $\V$-split and $B\in\Mod(\EE)$ then also $C\in\Mod(\EE)$.
	\end{enumerate} 
\end{prop}
\begin{proof}
	(1). Consider a weight $N\colon\C\to\V$ and a diagram $H\colon \C\to \Mod(\EE)$, and denote by $J\colon\Mod(\EE)\to\Str(\LL)$ the inclusion. We need to prove that $A:=\{N,JH\}\in\Str(\LL)$ is a model of $\EE$. 
	
	Given a judgement of the form $\dw t$ in $\EE$, by hypotheses $t$ is interpretable over $HC$ for any $C\in\C$; thus $t$ is also interpretable over $A$ by Lemma~\ref{limits of terms} above.
	
	Given an equation $(s=t)$ in $\EE$ we know that, for any $C\in\C$, the terms $s$ and $t$ are interpretable over $JH(C)$ and their interpretations coincide: $s_{JH(C)}=t_{JH(C)}$. Again by Lemma~\ref{limits of terms} it follows that $s$ and $t$ are also interpretable over $A=\{N,JH\}$ and their interpretations coincide; thus $A\in\Mod(\EE)$.
	
	(2). This is totally analogous to (1); one uses again Lemma~\ref{limits of terms} above. 
	
	(3). Let $m\colon A\to B$ a monomorphism in $\Str(\LL)$ which is right orthogonal in $\V$ with respect to every map in $\Gamma$. Assume moreover that $B\in\Mod(\EE)$; we need to prove that also $A$ is a model of $\EE$.
	
	To that end, it is enough to prove that if an $\LL$-term is interpretable over $B$ that it is also interpretable over $A$. Indeed, in that case for any equation $(s=t)$ in $\EE$ we can consider the diagram
	\begin{center}
		\begin{tikzpicture}[baseline=(current  bounding  box.south), scale=2]
			
			\node (z) at (-1,0) {$A^X$};
			\node (a) at (0,0) {$A^G$};
			\node (b) at (-1,-0.8) {$B^X$};
			\node (b1) at (0,-0.8) {$B^G$};

			\path[font=\scriptsize]
			
			(z) edge [->] node [left] {$m^X$} (b)
			(a) edge [->] node [right] {$m^G$} (b1)
			
			([yshift=-1.5pt]z.east) edge [->] node [below] {$s_A$} ([yshift=-1.5pt]a.west)
			([yshift=1.5pt]z.east) edge [->] node [above] {$t_A$} ([yshift=1.5pt]a.west)
			([yshift=-1.5pt]b.east) edge [->] node [below] {$s_B$} ([yshift=-1.5pt]b1.west)
			([yshift=1.5pt]b.east) edge [->] node [above] {$t_B$} ([yshift=1.5pt]b1.west);
		\end{tikzpicture}
	\end{center}
	where the squares with $s$ and $t$ respectively, commute. Thus, since $s_B=t_B$ and $m$ is a monomorphism, then also $s_A=t_A$. 
	
	We prove that interpretability of terms is transferred from $B$ to $A$ by induction on the rules defining terms. The fact for rules (1), (2), and (3) is trivial, since they don't change the interpretability status. Consider then $\tau:=s(t_{J,e})$ as in rule (4) and assume that $s$ and $t_j$, for $j\in J$, are interpretable over $A$ and $B$ (inductive hypothesis) and that that $\tau$ itself is interpretable over $B$. We can then form the solid part of the diagram below
	\begin{center}
		\begin{tikzpicture}[baseline=(current  bounding  box.south), scale=2]
			
			\node (a0) at (0,0.8) {$A^X$};
			\node (b0) at (1.2,0.8) {$A^Y$};
			\node (e0) at (2.4,0.8) {$A^{\sum G_j}$};
			\node (c0) at (0,0) {$B^X$};
			\node (d0) at (1.2,0) {$B^Y$};
			\node (f0) at (2.4,0) {$B^{\sum G_j}$};
			
			\path[font=\scriptsize]
			
			(a0) edge [dashed, ->] node [above] {} (b0)
			(a0) edge [->] node [left] {$m^X$} (c0)
			(b0) edge [->] node [left] {$m^Y$} (d0)
			(c0) edge [->] node [below] {} (d0)
			(b0) edge [>->] node [below] {$A^e$} (e0)
			(d0) edge [>->] node [above] {$B^e$} (f0)
			(e0) edge [->] node [right] {$m^{\sum G_j}$} (f0)
			
			(a0) edge [bend left=20, ->] node [above] {$ (t_j)_A$} (e0)
			(c0) edge [bend right=20, ->] node [below] {$ (t_j)_B$} (f0);
		\end{tikzpicture}	
	\end{center}  
	where, to conclude, we need to show the existence of the dashed arrow making the diagram commute. Transposing along the tensor-hom adjunction we obtain a commutative square in $\V$ as below.
	\begin{center}
		\begin{tikzpicture}[baseline=(current  bounding  box.south), scale=2]
			
			\node (a0) at (-1,0.8) {$(\sum G_j)\otimes A^X$};
			\node (b0) at (2,0.8) {$A$};
			\node (c0) at (-1,0) {$Y\otimes A^X$};
			\node (d0) at (0.8,0) {$Y\otimes B^X$};
			\node (e0) at (2,0) {$B$};
			
			\path[font=\scriptsize]
			
			(a0) edge [->] node [above] {$[(t_j)_A]^t$} (b0)
			(c0) edge [->] node [below] {$Y\otimes m^X$} (d0)
			(d0) edge [->] node [below] {$[(t_j)_B]^t$} (e0)
			(a0) edge [->] node [left] {$e\otimes A^X$} (c0)
			(b0) edge [->] node [right] {$m$} (e0);
		\end{tikzpicture}	
	\end{center}
	Since $m$ is right orthogonal to $e\otimes A^X$ by hypothesis, then the square above has a unique diagonal filler $Y\otimes A^X\to A$, which shows the interpretability of $\tau$ over $A$. 
	
	(4). Consider a $\V$-split morphism $e\colon B\twoheadrightarrow C$ in $\Str(\LL)$ with $B\in\Mod(\EE)$; thus there exists $r\colon C\to B$ in $\V$ with $e\circ r=1_C$. Given any judgement of the form $\dw t$ in $\EE$, for $t:(X,G)$, we need to prove that $t$ is interpretable over $C$: an easy calculation (arguing by induction on the construction of $t$) shows that that is indeed the case and $t_C$ is given by the composite
	$$ t_C\colon C^X\xrightarrow{\ r^X\ }B^X\xrightarrow{\ t_B\ } B^G\xrightarrow{\ e^G\ } C^G $$
	which is well defined since $t$ is interpretable over $B$. 
	
	If we are given an equation $(s=t)$ in $\EE$ then, by the arguments above we can consider the following commutative square 
	\begin{center}
		\begin{tikzpicture}[baseline=(current  bounding  box.south), scale=2]
			
			\node (z) at (-1,0) {$B^X$};
			\node (a) at (0,0) {$B^G$};
			\node (b) at (-1,-0.8) {$C^X$};
			\node (b1) at (0,-0.8) {$C^G$};

			\path[font=\scriptsize]
			
			(z) edge [->>] node [left] {$e^X$} (b)
			(a) edge [->>] node [right] {$e^G$} (b1)
			
			([yshift=-1.5pt]z.east) edge [->] node [below] {$s_B$} ([yshift=-1.5pt]a.west)
			([yshift=1.5pt]z.east) edge [->] node [above] {$t_B$} ([yshift=1.5pt]a.west)
			([yshift=-1.5pt]b.east) edge [->] node [below] {$s_C$} ([yshift=-1.5pt]b1.west)
			([yshift=1.5pt]b.east) edge [->] node [above] {$t_C$} ([yshift=1.5pt]b1.west);
		\end{tikzpicture}
	\end{center}
	where $e^X$ and $e^G$ are still split epimorphisms in $\V$. Since $s_B=t_B$ and $e^X$ is an epimorphism, it follows that also $s_C=t_C$; thus $C$ satisfies the equation $(s=t)$.
\end{proof}

As a corollary we can prove the following result. The notion of enriched factorization system that we consider below is the one treated in~\cite{lucyshyn2014enriched} (it consists of a orthogonal factorization system $(\E,\M)$ on $\V_0$ whose left class is stable under tensor product by objects of $\V$)

\begin{cor}\label{strong-closure}
	Let $\EE$ be a $\Gamma$-ary equational theory; then $\Mod(\EE)$ is closed in $\Str(\LL)$ under strong subobjects. More generally, if the maps in $\Gamma$ are contained in the left class of a proper enriched factorization system $(\E,\M)$ on $\V$, then $\Mod(\EE)$  is closed under $\M$-subobjects.
\end{cor}
\begin{proof}
	We apply Proposition~\ref{closure} above. Every strong monomorphism is, by definition, a monomorphism and right orthogonal to every epimorphism in $\V$. In particular it is right orthogonal to every $e\otimes Z$ for $e\in \Gamma$ and $Z\in\V$. 
	
	Similarly, if $\Gamma\subseteq\E$ where $(\E,\M)$ is a proper enriched factorization system on $\V$, then every map in $\M$ is a monomorphism (by properness) and right orthogonal to morphisms of the form $e\otimes Z$, for $Z\in\Gamma$, since $\E$ is closed under the tensor product (being enriched).
\end{proof}

It follows in particular that if all the maps in $\Gamma$ are strong epimorphisms, then $\Mod(\EE)$ is closed under subobjects in $\Str(\LL)$. Thus, when $\G=\V_\lambda$ and $\Gamma$ is made of isomorphisms (so that the notion of $\Gamma$-ary term coincides with that of~\cite{RT23EUA} by Example~\ref{iso-sigma}) we recover part of~\cite[Proposition~6.1]{RT23EUA}.

\begin{es}$ $\begin{enumerate}[leftmargin=1cm]
		\item Consider $\V=\Pos$ in the context of Example~\ref{Pos}. A strong monomorphism in $\Pos$ is a order reflecting map (that is, $m\colon A\to A'$ for which $a\leq b$ if and only if $m(a)\leq m(b)$). Thus, with Corollary~\ref{strong-closure} we recover \cite[Proposition~3.22]{adamek2021finitary}.
		\item If $\V=\Met$ and we are in the setting of Example~\ref{Met}. Then $\Gamma$ is contained in the class of surjections, which is the left class of the proper enriched factorization system (surjection, isometry). Thus, by Corollary~\ref{strong-closure} we know that models are closed under isometric subobjects, as stated in~\cite[Section~4]{mardare2016quantitative}.
		\item If $\V=\omega$-$\bo{CPO}$ and we are in the setting of Example~\ref{CPO}. Then $\Gamma$ is contained in the class of dense maps. These form the left class of an enriched proper factorization whose right class is given by the join-reflecting embeddings (order reflecting maps that also reflect joins of $\omega$-chains). Thus by Corollary~\ref{strong-closure} above, the models are closed under such subobjects, as stated in~\cite[Theorem~5.2]{adamek1985birkhoff}.
	\end{enumerate}
\end{es}

\begin{cor}\label{equat->monad}
	For any $\Gamma$-ary equational theory $\EE$ the forgetful $\V$-functor $$U\colon \Mod(\EE)\to \V$$ is strictly $\lambda$-ary monadic.
\end{cor}
\begin{proof}
	We know by \cite[Proposition~5.12]{RT23EUA} that the forgetful $\Str(\LL)\to \V$ is strictly $\Lambda$-ary monadic. This, plus Proposition~\ref{closure} implies that $U\colon \Mod(\EE)\to \V$ preserves limits, $\lambda$-filtered colimits, and strictly creates coequalizers of $\V$-split (that is, $U$-split) pairs. To conclude we only need to prove that it has a left adjoint, then the result follows by the strict monadicity theorem.
	
	To show that $U$ has a left adjoint it is enough to show that $\Mod(\EE)$ is locally presentable as a $\V$-category. First note that $\Str(\LL)$ is such (\cite[Proposition~3.5]{RT23EUA}). Now, by Proposition~\ref{closure} $\Mod(\EE)$ is closed in $\Str(\LL)$ under $\lambda$-filtered colimits and $\lambda$-pure subobjects (since these are regular, and hence strong, monomorphisms \cite[Proposition~2.31]{AR94:libro}); thus the underlying  category $\Mod(\EE)_0$ is accessible and accessibly embedded in $\Str(\LL)_0$ by \cite[Corollary~2.36]{AR94:libro}. By \cite[Corollary~3.23]{LT22:virtual} then $\Mod(\EE)$ is accessible as a $\V$-category; being also complete, it is then locally presentable.
\end{proof}

\subsection{Characterization theorem}

In this section we prove that $\V$-categories of models of $\Gamma$-ary equational theories coincide with $\V$-categories of algebras of $\lambda$-ary monads. We achieve this by comparing our $\Gamma$-ary equational theories with the $\lambda$-ary ones from~\cite{RT23EUA}, and then use the main theorem therein.

\begin{lema}\label{L-toL'-structures}
	For any $\lambda$-ary language $\mathbbm L$ there is an induced $\Gamma$-ary language $\LL'$ and a $\Gamma$-ary equational $\LL'$-theory $\EE$ for which $\Str(\LL)\cong\Mod(\EE)$.
\end{lema}
\begin{proof}
	It is enough to prove it for the case where $\LL$ consists of a single function symbol $f$ of arity $(X,Y)$. Consider $e\colon \tsum_{j\in J} G_j\twoheadrightarrow Y$ be in $\Gamma$, and define $\LL'$ to consist of function symbols $f_j:(X,G_j)$ for $j\in J$.
	Then $A$ is an $\LL$-structure if and only if it is an $\LL'$-structure that satisfies the judgement
	$$ \dw f_{J,e}$$
	for any $f\in\LL$.
	
	Indeed, an $\LL'$-structure $A$ comes with maps $(f_j)_A\colon A^X\to A^{G_j}$ and satisfies the judgement above if and only if $f_{J,e}$ is interpretable, if and only if $(f_{J,e})_A$ below exists.
	\begin{center}
		\begin{tikzpicture}[baseline=(current  bounding  box.south), scale=2]
			
			\node (a0) at (0,0.8) {$A^X$};
			\node (b0) at (1.2,0.8) {$A^Y$};
			\node (c0) at (1.2,0) {$A^{\sum_jG_j}$};
			
			\path[font=\scriptsize]
			
			(a0) edge [dashed, ->] node [above] {$(f_{J,e})_A$} (b0)
			(a0) edge [->] node [below] {$(f_j)_A\ \ \ \ \ \ $} (c0)
			(b0) edge [>->] node [right] {$A^e$} (c0);
		\end{tikzpicture}	
	\end{center}
	This say that $A$ is an $\LL$-structure with $f_A:=(f_{J,e})_A$. Conversely, every $\LL$-structure $A$ can be seen as an $\LL'$-structure by defining $(f_j)_A:=A^{e_j}\circ f_A$; this clearly satisfies the judgements above. 
\end{proof}

In the Lemma below by $(X,Y)$-ary $\LL$-term we mean one as in \cite[Definition~4.1]{RT23EUA}.

\begin{lema}\label{L-toL'-terms}
	Consider a $\lambda$-ary language $\LL$ and the induced $\LL'$ and $\EE$ as in the lemma above. 
	Given an $(X,Y)$-ary $\LL$-term $t$, with $X\in\V_\lambda$, constructed by applying the power rule only to $Z\in \G$. Then for any $e\colon\textstyle \sum_{j\in J} G_j\twoheadrightarrow Y$ in $\Gamma$ there exists a family $(t_j)_{j\in J}$ of $\Gamma$-ary $\LL'$-terms, where $t_j$ is $(X,G_j)$-ary, such that for any $A\in\Str(\LL')$\begin{itemize}[leftmargin=0.8cm]
		\item the family $t_{J,e}=(t_j)_{j\in J}$ is interpretable over $A$
		\item and $t_A= (t_{J,e})_A$.
	\end{itemize}
\end{lema}
\begin{proof}
	It is enough to prove that there exists an epimorphism $e$ and a family $(t_j)_{j\in J}$ as above for such $e$. Indeed, assuming this we can use rule (4) to get a family for any $e$: suppose we have $t_{J,e}=(t_j)_{j\in J}$ satisfying the two properties above. Given another $e'=(e'_j)_{j\in J'}\colon\textstyle \sum_{j\in J'} G'_j\twoheadrightarrow Y$ in $\Gamma$ we can apply rule (4) to form the $(X,G'_j)$-ary terms
	$$ t'_j:=e'_j(t_{J,e}) $$
	for any $j\in J'$. Then, an easy calculation shows that the family $t'_{J',e'}=(t'_j)_{j\in J'}$ satisfies the same two properties as $t_{J,e}$, but with respect to $e'$,
	
	We now proceed recursively on the construction of a $\lambda$-ary $\LL$-term following the rules of~\cite[Definition~4.1]{RT23EUA} that here we number as (1')-(4'). 
	
	(1'). The term $t$ is induced by a morphism $t\colon Y\to X$ in $\V$. By definition $\Gamma$ contains a map of the form $e=(e_j)_{j\in J}\colon\textstyle \sum_{j\in J} G_j\twoheadrightarrow Y$. Then we can take the $\Gamma$-ary $\LL'$-terms $t_j:= t(e_j)$ of arity $(X,G_j)$ as in Proposition~\ref{precomp-term}. By construction, $t_{J,e}$ is interpretable over any $\LL$-structure $A$ and satisfies $t_A= (t_{J,e})_A$.
	
	(2'). The term $t$ is induced by a function symbol $f:(X,Y)$ in $\LL$. Then by definition of $\LL'$ we have $e\colon \tsum_{j\in J} G_j\twoheadrightarrow Y$ in $\Gamma$ and a family of function symbols $f_j:(X,G_j)$, for $j\in J$, that satisfies the required properties (see the proof of Lemma~\ref{L-toL'-structures} above).
	
	(3'). Assume that we are given a term $t:(X,Y)$ together with a family $t_{J,e}$ of $\LL'$-terms satisfying the required properties, and an object $Z\in\G$. We need to prove that such a family exists also for the $\LL$-term $t^G$. For each $j\in J$ consider an epimorphism
	$$ e_j=(e_{j,h})_{h\in H_j}\colon \tsum_{h\in H_j} G_{j,h}\twoheadrightarrow Z\otimes G_j $$
	in $\Gamma$. Then, by definition of $\Gamma$, the composite
	$$ q\colon \sum_{j\in J,h\in H_j} G_{j,h}\xrightarrow{\ \sum_j e_j\ } \sum_{j\in J}Z\otimes G_j\xrightarrow{\ Z\otimes e\ } Z\otimes Y $$
	is still in $\Gamma$. Moreover, we can consider the $(Z\otimes X,G_{j,h})$-ary $\LL'$-terms $s_{j,h}=(t_j)^Z_{e_j,h}$ as per rule (3). Let $J*H=\{(j,h)|\ j\in J, h\in H_j\}$, it is easy to see that $(s_{J*H,q})$ satisfies the required properties for $t^Z$. 
	
	(4'). Consider now a family  $t_J=(t_j)_{j\in J}$, where each $t_j$ is an $(X_j,Y_j)$-ary term, and $s$ a $(\tsum_j Y_j,W)$-ary term. Assume by inductive hypothesis that the thesis holds for $s$ and each $t_j$, we need to prove it for $s(t_J)$. By hypothesis we have families, $s_{H,q}=(s_h)_{h\in H}$ for $s$ and $t^j_{K_j,e_j}=(t^j_k)_{k\in K_j}$ for each $t_j$. 
	
	Define $e:=\tsum_j e_j$, which still lies in $\Gamma$, and for each $h\in H$ the term
	$$ \sigma_h:= s_h(\hat t^J_{H,e}): (\tsum_j X_j, G_h) $$
	given as in Proposition~\ref{sum-domain}, where $G_h$ is the output arity of $s_h$. Then the family $\sigma_{H,q}$ satisfies the requirements for $s(t_J)$.
\end{proof}

As a consequence we can prove the following:

\begin{teo}\label{monad->equat}
	Let $T\colon \V\to\V$ be a $\lambda$-ary monad. Then there exists a $\Gamma$-ary equational theory $\EE$ on a $\Gamma$-ary language $\LL$ together with an isomorphism $E\colon \tx{Alg}(T)\to \Mod(\EE)$ making the triangle
	\begin{center}
		\begin{tikzpicture}[baseline=(current  bounding  box.south), scale=2]
			
			\node (a0) at (0,0.8) {$\tx{Alg}(T)$};
			\node (b0) at (1.3,0.8) {$\Mod(\EE)$};
			\node (c0) at (0.65,0) {$\V$};
			
			\path[font=\scriptsize]
			
			(a0) edge [->] node [above] {$E$} (b0)
			(a0) edge [->] node [left] {$U\ $} (c0)
			(b0) edge [->] node [right] {$\ U'$} (c0);
		\end{tikzpicture}	
	\end{center} 
	commute.
\end{teo}
\begin{proof}
	By \cite[Proposition~5.13]{RT23EUA} there exists a $\lambda$-ary equational theory $\EE'$ over a $\lambda$-ary language $\LL$ that satisfies the property above. By Lemma~\ref{L-toL'-structures} the $\V$-category of $\LL$-structures is isomorphic to the $\V$-category of models of some $\Gamma$-ary equational theory $\EE''$ on a $\Gamma$-ary language $\LL$. Then, by Lemma~\ref{L-toL'-terms}, each equation in $\EE$ between $\lambda$-terms is equivalent to a set of equations between $\Gamma$-ary terms. Thus we can expand $\EE''$ to form an equational $\Gamma$-ary theory $\EE$ whose $\V$-category of models is exactly $\Mod(\EE')$. This is enough to conclude.
\end{proof}

And thus:

\begin{teo}\label{char-single}
	The following are equivalent for a $\V$-category $\K$: \begin{enumerate}[leftmargin=1cm]
		\item $\K\simeq\Mod(\EE)$ for a $\Gamma$-ary equational theory $\EE$;
		\item $\K\simeq\tx{Alg}(T)$ for a $\lambda$-ary monad $T$ on $\V$;
		\item $\K$ is cocomplete and has a $\lambda$-presentable and $\V$-projective strong generator $G\in\K$;
		\item $\K\simeq\lambda\tx{-Pw}(\T,\V)$ is equivalent to the $\V$-category of $\V$-functors preserving $\lambda$-small powers, for some $\V_\lambda$-theory $\T$.
	\end{enumerate}
\end{teo}
\begin{proof}
	Put together \cite[Theorem~5.14]{RT23EUA} with Corollary~\ref{equat->monad} and Theorem~\ref{monad->equat} above.
\end{proof}

\subsection{Sound case}\label{souns-section}

We now move to the context of \cite{tendas2024more} where $\lambda$ filtered colimits are replaced by $\Phi$-flat colimits for a (weakly) sound class of weights $\Phi$.

Recall that, given a locally small class of weights $\Phi$, we say that a weight $M\colon\C\op\to\V$ is {\em $\Phi$-flat} if $M$-weighted colimits commute in $\V$ with $\Phi$-limits. The class $\Phi$ is called {\em weakly sound} if every $\Phi$-continuous weight $M\colon \C\op\to\V$, where $\C$ is $\Phi$-cocomplete $\C$, is $\Phi$-flat. See~\cite[Example~4.8]{LT22:limits} for a list of examples.

For the remainder of the section we fix a locally small and weakly sound class $\Phi$. As in \cite{tendas2024more} we denote by $\Phi\I$ the closure of the monoidal unit $I$ in $\V$ under $\Phi$-colimits. 

\begin{Def}
	A $\Gamma$-ary language $\LL$ is called {\em $(\Gamma,\Phi)$-ary} if all the function symbols in $\LL$ have domain arity in $\Phi\I$. A $\Gamma$-ary term $t$ is called {\em $(\Gamma,\Phi)$-ary} if its input arity is in $\Phi\I$. A $\Gamma$-ary equational theory $\EE$ over $\LL$ is called {\em $(\Gamma,\Phi)$-ary} if the terms appearing in it are all $(\Gamma,\Phi)$-ary.
\end{Def}

\begin{lema}\label{Phi-flat-pres}
	Let $\EE$ be a $(\Gamma,\Phi)$-ary equational theory over a $(\Gamma,\Phi)$-ary language. Then:\begin{enumerate}[leftmargin=1cm]
		\item the forgetful $\V$-functor $\Str(\LL)\to\V$ preserves $\Phi$-flat colimits;
		\item $\Mod(\EE)$ is closed in $\Str(\LL)$ under $\Phi$-flat colimits.
	\end{enumerate} 
\end{lema}
\begin{proof}
	Point (1) is given by \cite[Proposition~5.11]{tendas2024more}. for (2), we argue exactly as in Proposition~\ref{closure} by using Lemma~\ref{limits of terms} and that $\Phi$-flat colimits commute in $\V$ powers by the input arities appearing in $\EE$ (being $\Phi$-presentable objects).
\end{proof}

Then we prove the following result, which translates \cite[Theorem~5.13]{tendas2024more} into the context of $\Gamma$-ary equational theories.

\begin{teo}\label{char-single-Phi}
	The following are equivalent for a $\V$-category $\K$: \begin{enumerate}[leftmargin=1cm]
		\item $\K\simeq\Mod(\EE)$ for a $(\Gamma,\Phi)$-ary equational theory $\EE$;
		\item $\K\simeq\tx{Alg}(T)$ for a monad $T$ on $\V$ preserving $\Phi$-flat colimits;
		\item $\K$ is cocomplete and has a $\Phi$-presentable and $\V$-projective strong generator $G\in\K$;
		\item $\K\simeq\Phi\tx{-Pw}(\T,\V)$ is equivalent to the $\V$-category of $\V$-functors preserving $\Phi\I$-powers, for some $\Phi\I$-theory $\T$.
	\end{enumerate}
\end{teo}
\begin{proof}
	The last three equivalences are given by \cite[Theorem~5.13]{tendas2024more}. The implication $(1)\Rightarrow(2)$ follows from Corollary~\ref{equat->monad} and Lemma~\ref{Phi-flat-pres} above. 
	
	For $(2)\Rightarrow (1)$, by \cite[Proposition~5.12]{tendas2024more} there is a $\Phi$-ary equational theory $\EE'$ that satisfies $\K\simeq\Mod(\EE')$. Then by Lemma~\ref{L-toL'-terms} we can replace $\EE'$ with a $(\Gamma,\Phi)$-ary equational theory $\EE$ which has the same models (this is $(\Gamma,\Phi)$-ary because the input arities do not change).
\end{proof}

\section{The 2-categorical case}\label{2-cat-sect}

In this section we consider the case of $\V=\bo{Cat}$ with strong generator the set 
$$\G:=\{1=\{*\}, \bb 2 =\{0\to 1\}\};$$
just $\bb 2$ would have been enough, but having $1$ makes the languages more tractable, and makes it possible to eliminate part of the power rule for terms. For simplicity here we restrict only the finitary case.

Recall that $X\in\Cat$ is finitely presentable if and only if there exists a finite set $S$ of morphisms in $X$ that generates all maps of $X$ under composition (this in turn implies that $X$ has finitely many objects); it follows that we have an induced epimorphism
$$ e_S\colon\sum_{f\in S}\bb 2\to X $$
given by taking componentwise all the morphisms of $S$. The we can define $\Gamma$ to be the set of all such epimorphisms; it is easy to see that this satisfies the required closure properties.

We denote $(X,1)$-ary function symbols and terms with Latin letters $ f,t :X $
and call them {\em $X$-ary function symbols} and {\em $X$-ary terms}. Instead, we denote $(X,\bb 2)$-ary function symbols and terms with Greek letters $\sigma,\tau:X$ and call them  {\em $X$-ary 2-function symbols} and {\em $X$-ary 2-terms}. 

Every $X$-ary 2-term $\sigma$ can be written as
	$$ \sigma:\sigma_0\Rightarrow \sigma_1$$
where $\sigma_0:=j_0(\sigma)$ and $\sigma_1:=j_1(\sigma)$ are the $X$-ary terms obtained by applying the superposition rule to the terms given by the two object inclusions $j_0,j_1\colon 1\to\bb 2$. Note that if $\sigma$ is a 2-function symbol on a given language, we can assume without loss of generality (see Remark~\ref{cat-structures} below), that $\sigma_0$ and $\sigma_1$ are function symbols too. Under this, the definition of $\Gamma$-ary language becomes:

\begin{Def}
	A $\Gamma$-ary language $\LL$ amounts to:\begin{enumerate}[leftmargin=1cm]
		\item a set $\LL_1$ of function symbols $f:X$ with $X\in\Cat_f$;
		\item a set $\LL_2$ of 2-function symbols $\sigma\colon f\Rightarrow g$ of arity $X\in\Cat_f$ between function symbols $f,g\in\LL_1$ of the same arity.
	\end{enumerate}
\noindent An $\LL$-structure $A$ is the data of an object $A\in\Cat$ together with:\begin{enumerate}[leftmargin=1cm]
	\item a functor $f_A\colon A^X\to A$ for any $X$-ary $f\in\LL_1$;
	\item a natural transformation $\sigma_A\colon f_A\Rightarrow g_A\colon A^X\to A$ for any $X$-ary $\sigma\colon f\Rightarrow g$ in $\LL_2$.
\end{enumerate}
\end{Def}

\begin{obs}\label{cat-structures}
	This notion of structure differs from the one of Section~\ref{structures-and-terms} because, when interpreting a 2-function symbols $\sigma\colon f\Rightarrow g$, we are implicitly asking an $\LL$-structure $A$ to satisfy the equations
	$$ (j_0(\sigma)=f)\ \ \ \text{ and }\ \ \ (j_1(\sigma)=g).$$
	This will not change our theory in any way; however, it will make it somewhat easier to write-down examples.	 
\end{obs}

Then, taking into account Remark~\ref{lessrules} and that we can skip rule (3) for 2-terms since it will be shown to be redundant (see Remark~\ref{power-2-term} below), the notion of term is generated as follows.

\begin{Def}\label{cat-term}
	The class of $\Gamma$-ary $\LL$-terms is then defined recursively as follows:\begin{enumerate}[leftmargin=1cm]
		\item[(1)] Given $X\in\Cat_f$, every object $x\in X$ is an $X$-ary term and every morphism $(\rho\colon x\to x')\in X$ is an $X$-ary 2-term; 
		\item[(2)] Every $X$-ary function symbol $f\in\LL_1$ is an $X$-ary term, and every $X$-ary 2-function symbol $\sigma\in\LL_2$ is an $X$-ary 2-term;
		\item[(3)] If $t$ is a $X$-ary term, then $t^{\bb 2}$ is a $(\bb 2\times X)$-ary 2-term;
		%\item[(3.2)] If $t$ is a $X$-ary 2-term, then $t^{\bb 2}_h$ is a $(\bb 2\times X)$-ary 2-term for any $h\colon \bb 2\to \bb 2\times \bb 2$;
		\item[(4)] Given a $Y\in\Cat_f$ and a finite set $S$ of generating morphisms for $Y$ (that is, an arrow in $\Gamma$). For a $Y$-ary 2-term (resp. term) $\sigma$, and a family  $(\tau_{x_1},\dots, \tau_{x_n})$ of $X$-ary 2-terms indexed on the elements of $S$; then $\sigma(\tau_{x_1},\dots,\tau_{x_n})$ is a $X$-ary 2-term (resp. term).
	\end{enumerate}
	
\end{Def}

\begin{nota}
	For practical purposes we will sometimes denote natural transformations $\sigma\colon f\Rightarrow g$, for $f,g\colon B\to A$, as maps $B\to A^{\bb 2}$, rather than consider their transposes.
\end{nota}

\begin{Def}
	The interpretation of a $\Gamma $-ary term over an $\LL$-structure $A$ is then defined as follows:
	\begin{enumerate}[leftmargin=1cm]
		\item[(1)] Given $X\in\Cat_f$, the $X$-ary terms $x\in X$ are always interpretable in $A$ and their interpretation is given by the evaluation functor
		$$ x_A\colon A^X\xrightarrow{\tx{ev}_x}A.$$
		Similarly, every $X$-ary 2-term $(\rho\colon x\to x')\in X$ is interpretable over $A$ and its interpretation is given by the natural transformation
		$$(\rho_A\colon x'_A\Rightarrow x_A)\colon A^X\xrightarrow{\tx{ev}_\rho} A^{\bb 2}$$
		induced by evaluating at $\rho$.
		\item[(2)] Every $X$-ary function symbol $f\in\LL_1$ is interpretable over $A$ and its interpretation is given by the map $$f_A\colon A^X\to A$$ obtained by the fact that $A$ is an $\LL$-structure. Similarly for $X$-ary 2-function symbol in $\LL_2$.
		\item[(3)] If $t$ is an $X$-ary term that is interpretable over $A$, then $t^{\bb 2}$ is interpretable over $A$ and its interpretation is given by the  natural transformation
		$$ t^{\bb 2}_A\colon A^{\bb 2\times X}\xrightarrow{\ \cong\ } (A^X)^{\bb 2}\xrightarrow{(t_A)^{\bb 2}} A^\bb 2. $$
		\item[(4)] Given $(Y,S)$ in $\Gamma$. For a $Y$-ary 2-term $\sigma$, and a family  $(\tau_{x_1},\dots \tau_{x_n})$ of $X$-ary 2-terms indexed on the elements of $S$; then 
		$\sigma(\tau_{x_1},\dots \tau_{x_n})$ is interpretable over $A$ if and only if $\sigma$ and each $\tau_{x_i}$ are interpretable over $A$, and the family $(\tau_{x_i})_A\colon A^{X}\to A^\bb 2$ assembles into a (uniquely determined) functor $(\tau_{\bar x})_A$ as below.
		\begin{center}
			\begin{tikzpicture}[baseline=(current  bounding  box.south), scale=2]
				
				\node (a0) at (0,0.8) {$A^X$};
				\node (b0) at (1.2,0.8) {$A^Y$};
				\node (c0) at (1.2,0) {$A^{\sum_jG_j}$};
				
				\path[font=\scriptsize]
				
				(a0) edge [dashed, ->] node [above] {$(\tau_{\bar x})_A$} (b0)
				(a0) edge [->] node [below] {$(\tau_{x_i})_A\ \ \ \ \ \ $} (c0)
				(b0) edge [>->] node [right] {$A^{e_S}$} (c0);
			\end{tikzpicture}	
		\end{center}
		In that case the interpretation of $\sigma(\tau_{x_1},\dots \tau_{x_n})$ is given by the composite
		$$ \sigma(\tau_{x_1},\dots \tau_{x_n})_A\colon A^{X}\xrightarrow{\ (\tau_{\bar x})_A  \ } \textstyle A^Y\xrightarrow{\ \sigma_A\ } A^{\bb 2}.$$
		The same applies when $\sigma$ is a term rather than a 2-term.
	\end{enumerate}
\end{Def}

\begin{obs}\label{sample-terms}
	The following examples of terms will be useful later on and give an idea of how interpretability works:\begin{enumerate}[leftmargin=1cm]
		\item For any $X$-ary 2-term $\sigma$ we have two $X$-ary terms $\sigma_0$ and $\sigma_1$ defined by 
		$$ \sigma_i:=j_i(\sigma) $$
		where $j_i\colon 1\to\bb 2$ picks the point $i\in\{0,1\}$. For any $\A\in\Str(\LL)$ for which $\sigma$ is interpretable, then also $\sigma_0$ and $\sigma_1$ are interpretable, and their interpretation coincides respectively with the domain and codomain of the natural transformation $\sigma_A$.
		\item Given an $X$-ary term $t$ we have an $X$-ary 2-term $1_t\colon t\Rightarrow t$ defined by 
		$$1_t:= !(t) $$
		where $!\colon \bb 2\to 1$ is the unique map. If $A\in\Str(\LL)$ we have $(1_t)_A=1_{t_A}$.
		\item Given two $X$-ary 2-terms $\sigma$ and $\tau$, we have a new $X$-ary 2-term $\sigma\circ \tau$ defined by
		$$ \sigma\circ \tau:=j_{02}(\sigma,\tau)_S $$
		where we are applying rule (4) to $Y=\mathbb 3 =\{0\to 1\to 2\}$, with generating family $S$ the two arrows $0\to 1$ and $1\to 2$, and with $j_{02}$ the $\mathbb 3$-ary 2-term given by the arrow $0\to 2$ in $\mathbb 3$. If $A\in\Str(\LL)$, then $\sigma\circ \tau$ in interpretable if both $\sigma$ and $\tau$ are, and moreover $(\tau_1)_A=(\sigma_0)_A$. In that case then $(\sigma\circ \tau)_A=\sigma_A\circ \tau_A$.
		\item Given an $X$-ary 2-term $\sigma$ we have an $X$-ary 2-term $\sigma^{-1}$ defined by
		$$ \sigma^{-1}:= j_{10}(\sigma)_{\{0\to 1\}} $$
		where we are applying rule (4) to $Y=\bb I=\{0\cong 1\}$, with generating family the arrow $0\to 1$, and with $j_{10}$ the $\bb I$-ary 2-term given by the arrow $1\to 0$ of $\bb I$. If $A\in\Str(\LL)$, then $\sigma^{-1}$ is interpretable if $\sigma$ is and $\sigma_A$ is an invertible 2-cell; in that case then $(\sigma^{-1})_A=\sigma_A^{-1}$.
	\end{enumerate}
\end{obs}

\begin{obs}\label{power-2-term}
	Here we explain why in the definition of term we do not need to apply the power rule to 2-terms. Given an $X$-ary 2-term $\tau\colon t\Rightarrow s$ the power terms generated by it are those of the form $t^{\bb 2}_h$ for any (non trivial) $h\colon \bb 2\to \bb 2\times \bb 2$. Then it is easy to see that:\begin{enumerate}[leftmargin=1cm]
		\item if $h$ picks the arrow $(0,0)\to(0,1)$, then $\tau_h^{\bb 2}\equiv t^\bb2$;
		\item if $h$ picks the arrow $(1,0)\to(1,1)$, then $\tau_h^{\bb 2}\equiv s^\bb2$;
		\item if $h$ picks the arrow $(0,0)\to(1,0)$, then $\tau_h^{\bb 2}\equiv \tau(\pi_0)$ --- see below;
		\item if $h$ picks the arrow $(0,1)\to(1,1)$, then $\tau_h^{\bb 2}\equiv \tau(\pi_1)$ --- see below;
		\item if $h$ picks the arrow $(0,0)\to(1,1)$, then $\tau_h^{\bb 2}$ is equivalent to the composite (in the sense of Example~\ref{sample-terms}(3) above) of the terms in point (1) and (4) (or equivalently, (2) and (3)).
	\end{enumerate}
	The terms $\tau(\pi_i)$, for $i=0,1$, are obtained by applying Proposition~\ref{precomp-term} to the morphisms $\pi_i\colon X\to \bb 2\times X$ in $\Cat_f$ which send $x\in\ X$ to $(i,x)\in\bb 2\times X$.
	%Where $\pi_0$ (and similarly $\pi_1$) is the family  of $(\bb2\times X)$-ary 2-terms defined as follows: fix a finite family of morphisms $S$ such that $(X,S)\in\Gamma$, then 
	%$$ \pi_0=(\iota_{0\times f})_{f\in S} $$
	%where $\iota_{0\times f}$ is the $(\bb2\times X)$-ary 2-term corresponding to the morphism $0\times f$ in $\bb 2\times X$.\\
	Thus, it follows that all the power terms listed above are redundant, as their equivalent reformulation on the right-hand-side does not use the power rule on 2-terms.
\end{obs}

\begin{es}
	Let $\LL$ be the empty language, so that $\Str(\LL)=\Cat$. Consider the $\bb 2$-ary 2-term $1_\bb 2$ induced by the identity on $\bb 2$. Then a category $A$ satisfies 
	$$ \dw 1_{\bb 2}^{-1} $$
	if and only if every morphism of $A$ is invertible, if and only if $A$ is a groupoid.  
\end{es}

\begin{es}
	Over the empty language consider the $\bb 2$-ary term $ !:=(1_\bb 2)_{\{\tx{id}\}}$,where $\{\tx{id}\}$ is a generating family for the singleton category $1$. Then a category $A$ satisfies 
	$$\dw\ !$$ 
	if and only if every morphism in it is an identity morphism, if and only if it is discrete.
\end{es}

In the next example we treat the case of categories with specified limits (or colimits) of a given shape. When defining the 2-function symbols of the language we use terms that are constructed starting from the 1-dimensional function symbols (see for instance the codomains of the 2-function symbols in (b) and (c) below). This should simply be interpreted as a shortcut: it replaces and additional 1-dimensional function symbol, and an additional equation between that and the term in question.

\begin{es}[Categories with specified (co)limits]\label{limits}
	Given a category $D\in\Cat_f$, we will now present an equational theory for the 2-category of small categories with chosen $D$-limits and functors preserving them on the nose (this can be easily generalized to the case of colimits, or when $C$ is not finitely presentable using an infinitary theory).
	
	Given $D\in\Cat_f$, fix a finite set $S_D$ of generating morphisms for it. Consider the category $0*D$ obtained by adding an initial object to $D$; then the set $S_D\cup\{0_d\colon 0\to d \}_{d\in D}$ is a finite generating set for $0*D$. Similarly, we can consider the category $\bb 2*D$ obtained by adding an arrow $0'\to 0$ to $0*D$.
	
	Let $\LL$ be the language consisting of:\begin{itemize}[leftmargin=0.8cm]
		\item[(a)] a function symbol $$\lim$$ of arity $D$;
		\item[(b)] for any $d\in\D$, a 2-function symbol $$\tx{pr}_d\colon \lim\Rightarrow i_d$$ of arity $D$, where $i_d\colon 1\to D$ picks the object $d\in D$;
		\item[(c)] a 2-function symbol $$\rho\colon i_0\Rightarrow \lim(\iota_D)$$ of arity $0*D$, where $i_0\colon 1\to 0*D$ and $\iota_D\colon D\to 0*D$ are the inclusions.
	\end{itemize}
	Over this language consider the theory $\EE$ defined by:\begin{enumerate}[leftmargin=1cm]
		\item The judgement
		$$ \dw (\tx{pr}_d,f)_{\{d\in D,\ f\in S_D\}} $$
		indexed over the generating family of $0*D$ introduced above.
		\item The equations
		$$ \tx{pr}_d(\iota_D)\circ \rho=0_d $$
		for any $d\in D$, where: $0_d$ is the 2-term given by the arrow $0\to d$ in $0*D$, and $\circ$ is the operation defined in Example~\ref{sample-terms}(3).
		\item Finally, for any $d\in D$ consider the $(\bb 2*D)$-ary 2-term $(\tx{pr}_d\circ 1_\bb2)$ constructed in a similar way to that of Example~\ref{sample-terms}(3). Then define the family
		$$ \tx{pr}\circ 1_\bb2 := (\tx{pr}_d\circ 1_\bb2,f)_{\{d\in D,\ f\in S_D\}}$$
		indexed over the generating family of $0*D$. Then $\EE$ contains the equation
		$$ \rho(\tx{pr}\circ 1_\bb2)=\iota_\bb2 $$
		where $\iota_\bb2$ is the 2-term corresponding to the arrow $0'\to 0$ in $\bb 2*D$.
	\end{enumerate}
	Given a category $A$ together with a choice of $D$-limits, then we have a functor
	$$ \textstyle\lim_A\colon A^D\to A$$
	that associated to any diagram of shape $D$ its chosen limit, together with natural transformations
	$$ (\tx{pr}_d)_A\colon \textstyle\lim_A\Rightarrow A^{i_d}\colon A^D\to A $$
	giving the components of the limit cone associated to a diagram, as well as a natural transformation
	$$ \rho_A\colon (i_0)_A\Rightarrow \textstyle\lim_A\circ A^{\iota_D}\colon A^{0*D}\to A $$
	that associates to any cone over a diagram of shape $D$, the unique map into the limit. This makes $A$ into an $\LL$-structure. As such, $A$ satisfies the judgement in (1) since the components of the limit cone form indeed a cone over the diagram. It satisfies (2) since, given a cone, composing the induced map into the limit with the limit cone, gives back the cone we started with. And finally $A$ satisfies (3) since, given any morphism $g$ into $\textstyle\lim_A(H)$, the image through $\rho_A$ of the cone obtained by composing $g$ with the limit cone, is $g$ itself (by uniqueness of the factorization). Thus $A\in\Mod(\EE)$. It is easy to see that conversely any model of $E$ induces a choice of $D$-limits (where (2) expresses the fact the $\rho$ gives a factorization through the limit cone, and (3) says that such factorization is unique).
\end{es}

\subsection{Epimorphisms and strong monomorphisms in $\Cat$}$ $

The characterization of epimorphisms in $\Cat$ can be found in~\cite{isbell1968epimorphisms}; that of strong monomorphisms is probably all folklore. In this section we recall such notions ellipticity as they will be needed to prove the Birkhoff variety theorem of Section~\ref{birk}.

\begin{Def}
	Consider a functor $f\colon A\to B$ in $\Cat$. We define $\overline{f(A)}$ to be the smallest subcategory of $B$ such that:\begin{itemize}[leftmargin=0.8cm]
		\item $f(a)\in \overline{f(A)}$ for any $a\in A$, and $f(h)\in\overline{f(A)}$ for any $h\colon a\to a'$ in $A$;
		\item if $k\colon b\to b'$ is in $\overline{f(A)}$ and $k$ is invertible in $B$, then $k^{-1}\in \overline{f(A)}$. 
	\end{itemize}
\end{Def}

\begin{prop}
	Let $f\colon A\to B$ be a morphism in $\Cat$; then:\begin{enumerate}[leftmargin=1cm]
		\item $f$ is an epimorphism if and only if $\overline{f(A)}=B$;
		\item $f$ is a strong monomorphism if and only if it is injective on objects, faithful, and conservative.
	\end{enumerate}
	The (epi, strong mono) factorization of a morphism $f\colon A\to B$ is given by 
	\begin{center}
		\begin{tikzpicture}[baseline=(current  bounding  box.south), scale=2]
			
			\node (a0) at (0,0.6) {$A$};
			\node (b0) at (1.6,0.6) {$B$};
			\node (c0) at (0.8,1.1) {$\overline{f(A)}$};
			
			\path[font=\scriptsize]
			
			(a0) edge [->] node [below] {$f$} (b0)
			(a0) edge [->>] node [above] {$e\ \ $} (c0)
			(c0) edge [>->] node [above] {$\ \ m$} (b0);
		\end{tikzpicture}	
	\end{center}
	where $e$ is the codomain restriction of $f$, and $m$ is the inclusion of $\overline{f(A)}$ into $B$.
\end{prop}
\begin{proof}
	The final part will follow from the first two points since the codomain restriction $A\to \overline{f(A)}$ of $f$ is an epimorphism, and the inclusion $\overline{f(A)}\to B$ is injective on objects, faithful, and conservative.
	
	Point (1) is in \cite{isbell1968epimorphisms}. % , but we give a proof for completeness. Assume first that $\overline{f(A)}=B$, then it's clear that any functor out of $B$ is determined by its values at the objects and arrows that are in the image of $f$. Thus $f$ is an epimorphism. 
	For (2), if $f$ is a strong monomorphism then it is injective on objects and faithful (since it is a monomorphism) and is conservative since it is (by definition) right orthogonal to the epimorphism
	$$ \bb 2:=\{0\to 1\} \longrightarrow \bb I:=\{0 \cong 1 \}. $$
	Assume now that $f$ is injective on objects, faithful, and conservative; consider its (epi, strong mono) factorization given by $q\colon A\to C$ and $m\colon C\to B$. Then by (1) we have $\overline{q(A)}=C$, and (since $m$ is a strong mono) also $\overline{q(A)}\cong \overline{f(A)}$. But it is easy to see that, since $f$ is injective on objects, faithful, and conservative, then $\overline{f(A)}\cong A$. It follows that $\overline{q(A)}\cong A$ and hence that $q$ is an isomorphism. Thus $f$ is a strong monomorphism.
\end{proof}

\subsection{Free structures}\label{free-struct}$ $

Let us fix a $\Gamma$-ary language $\LL$ over $\Cat$, we shall now give an explicit construction of the free $\LL$-structure on $X\in\Cat_f$. The notion of equivalence between terms mentioned below is the one from Definition~\ref{equivalent}.

\begin{Def}
	Given $X\in\Cat_f$, denote by $FX$ the category which has:\begin{itemize}[leftmargin=0.8cm]
		\item[(i)] {\em objects}: the set of $X$-ary terms $s$ which are interpretable over any $\LL$-structure, quotient by the equivalence relation $\equiv$ between terms;
		\item[(ii)] {\em arrows}: the set of $X$-ary 2-terms $\sigma$ which are interpretable over any $\LL$-structure, quotient by the equivalence relation $\equiv$ between 2-terms;
	\end{itemize}
	where an arrow $[\sigma]$ has domain $[\sigma_0]$ and codomain $[\sigma_1]$. Identities and composition are defined as in Remark~\ref{sample-terms}
\end{Def}

The next step is to make $FX$ into an $\LL$-structure. Before doing so we need to set some notation. 

\begin{nota}\label{functor->term}
	Given a $Y$-ary (2-)term $\sigma$ in $\LL$, a functor $h\colon Y\to FX$, and a finite family $S_Y=\{y_i\}_{i\leq n}$ of generating morphisms for $Y$, we consider the $X$-ary (2-)term
	$$ \sigma(h):= \sigma(h(y_1),\cdots,h(y_n)) $$
	where, for each $i\leq n$, we have chosen a representative of the morphism $[h(y_i)]$ in $FX$.
\end{nota}

Given any $Y$-ary function symbol $f\in\LL$ we define the interpretation of $f$ as the functor
$$ f_{FX}\colon FX^Y\to FX $$
sending an object $h\colon Y\to FX$ to $f_{FX}(h):=[f(h)]$, and an arrow $\eta\colon h_0\Rightarrow h_1\colon Y\to FX$ to $f_{FX}(\eta):=[f^{\bb 2}(\eta)]$
where we see $\eta$ as a map $\eta\colon \bb 2\times Y\to FX$. \\
Similarly, given any $Y$-ary 2-function symbol $\sigma\in\LL$ with define the interpretation of $\sigma$ as the functor
$$ \sigma_{FX}\colon FX^Y\to FX^{\bb 2} $$
sending an object $h\colon Y\to FX$ to $\sigma_{FX}(h):=[\sigma(h)]$, and an arrow $\eta\colon h_0\Rightarrow h_1\colon Y\to FX$ to the commutative square (that is, an arrow in $FX^{\bb 2}$)
\begin{center}
	\begin{tikzpicture}[baseline=(current  bounding  box.south), scale=2]
		
		\node (a0) at (0,0.8) {$\bullet$};
		\node (b0) at (1,0.8) {$\bullet$};
		\node (c0) at (0,0) {$\bullet$};
		\node (d0) at (1,0) {$\bullet$};
		
		\path[font=\scriptsize]
		
		(a0) edge [->] node [above] {$[\sigma(h_0)]$} (b0)
		(a0) edge [->] node [left] {$[\sigma(\pi_0)(\eta)]$} (c0)
		(b0) edge [->] node [right] {$[\sigma(\pi_i)(\eta)]$} (d0)
		(c0) edge [->] node [below] {$[\sigma(h_1)]$} (d0);
	\end{tikzpicture}	
\end{center} 
where, again, we see $\eta$ as a map $\bb2\times Y\to FX$, and to define the 2-terms $\sigma(\pi_i)$ we follow the notation of Remark~\ref{power-2-term}. \\
It is easy to see that such assignments are functorial and don't depend on the choices made in Notation~\ref{functor->term}; thus $FX\in\Str(\LL)$ with the structure defined above.

Note that there is a functor $\eta_X\colon X\to UFX$ which sends objects and arrows of $X$ to the corresponding (equivalence classes of) $X$-ary terms and 2-terms as per Definition~\ref{cat-term}(1).

\begin{teo}\label{free-structures-teo}
	Let $\LL$ be a language over $\Cat$, and $X\in\Cat_f$. Then the $\LL$-structure $FX$ is the value at $X$ of the left adjoint to the forgetful 2-functor $U\colon\Str(\LL)\to\Cat$; the unit of the adjunction is given by the functor $\eta_X\colon X\to UFX$ defined above.
\end{teo}
\begin{proof}
	Pre-composition by $\eta_X$ induces a map
	$$ \Str(\LL)(FX,A)\longrightarrow\Cat(X,UA); $$
	we need to prove that it is an isomorphism. Since $\Str(L)$ has powers and $U$ preserves them, it is enough to show that this is a bijection of sets. Or equivalently, that for each $k\colon X\to UA$ there exists a unique morphism of $\LL$-structures $\hat k\colon FX\to A$ for which $U(\hat k) \circ \eta_X= k$. 
	
	Fix a functor $k\colon X\to UA$, we construct a morphism of $\LL$-structures $\hat k\colon FX\to A$ by setting
	$$ \hat k([s]):=s_A(k) \ \ \ \  \text{ and } \ \ \ \ \hat k([\sigma]):=\sigma_A(k) $$
	for all $X$-ary terms $s$ and 2-terms $\sigma$ that represent the objects and morphisms of $FX$. This is independent from the choice of the representatives by definition of equivalence between terms, and is functorial by how composition of terms and interpretations are defined. We need to show that $\hat k$ is a morphism of $\LL$-structures; meaning that for any (2-)function symbol $s\in\LL$ the equality
	$$ \hat k\circ s_{FX}= s_A\circ \hat k^Y $$
	holds (for a 2-function symbol we replace $\hat k$ on the left with $\hat k^\bb 2$, the proof does not change). On the one hand, given $h\colon Y\to FX$ we know that
	\begin{equation}\label{morofL-str}
		\hat k\circ s_{FX}(h)= \hat k([s(h)])=(s(h))_A(k)
	\end{equation}
	where $s(h)$ is the $Y$-ary term defined as in Notation~\ref{functor->term}. On the other hand
	$$ s_A\circ \hat k^Y(h)= s_A(\hat k\circ h); $$
	but $\hat k\circ h\colon Y\to A$ is, by definition of $\hat k$, the value at $k$ of
	$$ A^X\xrightarrow{\ h(\bar y)_A\ }A^Y $$
	and this, post-composed with $s_A$, gives exactly $(s(h))_A(k)$. Thus the equality~\ref{morofL-str} holds on objects; the same strategy works for morphisms simply by seeing them as maps $\eta\colon\bb 2\times Y\to FX$.
	
	Finally, note the equality $U(\hat k) \circ \eta_X= k$ holds by definition of $\eta_X$, and $\hat k$ is the only one with this property: any morphism of $\LL$-structures $FX\to A$ is completely determined by the assignment of the variable terms in Definition~\ref{cat-term}(1); indeed, the interpretation of every other term in $FX$ is determined by the variable terms and from it being a morphism of $\LL$-structures.
\end{proof}

\subsection{Birkhoff variety theorem}\label{birk}$ $

In this final section we characterize the full subcategories of $\Str(\LL)$ that arise as 2-categories of models of equational theories under certain closure properties.

\begin{nota}\label{epi->term}
	Consider an epimorphism of the form $e\colon FX\to Z$ in $\Cat$, with $X\in\Cat_f$ and domain the underlying category of $FX\in\Str(\LL)$. \begin{itemize}[leftmargin=0.8cm]
		\item[(i)] For any morphism $f\in Z$ fix a finite family $([\sigma_1],\cdots,[\sigma_n])$ of morphisms in $FX$ (represented by $X$-ary 2-terms) such that $f$ can be written by composing and inverting the image of those through $e$. Let $\tau^e_f$ be the $X$-ary 2-term obtained from $(\sigma_1,\cdots,\sigma_n)$ by applying the same composition and inversion rules (in the same order) in the sense given by Remark~\ref{sample-terms}(3,4).
		\item[(ii)] Consider an $Y$-ary (2-)term $\sigma$ over $\LL$ and a functor $h\colon Y\to Z$. Fix a finite family $S_Y=\{y_i\}_{i\leq n}$ of generating morphisms for $Y$, we can then consider the $X$-ary (2-)term
		$$ \sigma(h):= \sigma(\tau^e_{hy_1},\cdots,\tau^e_{hy_n}) $$
		where, for each $i\leq n$, we have used the terms introduced in (i) above.
	\end{itemize}

\end{nota}

\begin{teo}\label{birk-theo}
	Let $\LL$ be a finitary language. The following are equivalent for a full subcategory $\A$ of $\Str(\LL)$:\begin{enumerate}[leftmargin=1cm]
		\item $\A=\Mod(\EE)$, for some finitary equational theory $\EE$ on $\LL$;
		\item $\A$ is closed un $\Str(\LL)$ under products, powers, strong subobjects, $\V$-split quotients, and filtered colimits.
	\end{enumerate}
\end{teo}
\begin{proof}
	$(1)\Rightarrow (2)$ follows from Proposition~\ref{closure} and Corollary~\ref{strong-closure}. Assume conversely that $\A$ is closed un $\Str(\LL)$ under products, powers, strong subobjects, and filtered colimits. Then, by the abstract Birkhoff theorem~\cite[Chapter 3, 3.4]{manes2012algebraic}, the ordinary inclusion $J_0\colon\A_0\to\Str(\LL)_0$ has a left adjoint $L_0$ and for any $X\in\Cat$, the unit $\gamma_X\colon FX\to F'X:=J_0L_0FX$ is an epimorphism in $\Cat$ (where $F$ is the functor taking free $\LL$-structures). Since $\A$ is closed under powers it follows that the left adjoint is actually enriched. 
	
	Now we will see that $\A$ is the (ordinary) orthogonality class defined by the set $\{\gamma_X\}_{X\in\Cat_f}$; that is, given an $\LL$-structure $A$, then  $A\in\A$ if and only if 
	$$ \Str(\LL)_0(\gamma_X,A)\colon \Str(\LL)_0(F'X,A)\to \Str(\LL)_0(FX,A) $$
	is a bijection for any $X\in\Cat_f$ (the stronger enriched property actually holds, but the underlying one will be enough). On the one hand, if $A\in\A$ then $\Str(\LL)(\gamma_X,A)$ is a bijection for any $X\in\Cat_f$ by the universal property of the adjunction. \\
	Conversely, given $A$ that is orthogonal to any $\gamma_X$, for $X\in\Cat_f$, then $A$ is also orthogonal with respect to $\gamma_Y$ for any $Y\in\Cat$; indeed, since $J$ and $L$ preserve filtered colimits and $Y\cong \textstyle\colim_i X_i$ is a filtered colimit of finitely presentable objects, then $\gamma_Y\cong\textstyle\colim_i\gamma_{x_i}$, and thus $\Str(\LL)_0(\gamma_Y,A)$ is a limit of bijections, an therefore a bijection itself. Now, let $Y=UA$ be the underlying object of the $\LL$-structure $A$. Then, $A$ is a $\V$-split quotient of $F(UA)$ through the counit $F(UA)\to A$ (with the splitting in $\V$ induced by $\eta_{UA}$). Since $A$ is orthogonal with respect to $\gamma_{UA}$ it follows that it is also a $\V$-split quotient of $F'(UA)$, which lies in $\A$; thus $A\in\A$ by closure under $\V$-split quotients. 

	Next, using how each $FX$ is explicitly defined and how epimorphisms are presented in $\Cat$, we define the equational theory $\EE$ as follows: for any $X\in\Cat_f$\begin{enumerate}[leftmargin=1cm]
		\item If $\sigma,\tau$ are $X$-ary (1- or 2-)terms such that $\gamma_X([\sigma])=\gamma_X([\tau])$, then
		$$(\sigma=\tau)\in\EE.$$
		\item For any morphism $f\in F'X$ we set 
		$$ \dw \tau^{\gamma_X}_f\in\EE, $$
		where we follow Notation~\ref{epi->term}(i). 
		\item Given a function symbol $f\in\LL$ of arity $Y\in\Cat_f$, for any $h\colon Y\to F'X$ in $\Cat$ we consider the $X$-ary term $f(h)$. Let now $t_h$ be an $X$-ary term such that $\gamma_X([t_h])= f_{F'X}(h)$; then we set
		$$ (f(h)= t_h) \in\EE.$$ 
		\item Given a function symbol $f\in\LL$ of arity $Y\in\Cat_f$, for any $\eta\colon h\Rightarrow h'\colon Y\to F'X$ in $\Cat$ we consider the $X$-ary 2-term 
		$f^\bb 2(\eta)$, where we see $\eta$ as a map $\bb 2\times Y\to F'X$.
		Then
		$$ (f^\bb 2(\eta)= \tau^{\gamma_X}_{f_{F'X}(\eta)}) \in\EE.$$  
		\item Given a 2-function symbol $\sigma\in\LL$ of arity $Y\in\Cat_f$, for any $h\colon Y\to F'X$ in $\Cat$ we consider the $X$-ary 2-term $\sigma(h)$. Then
		$$ (\sigma(h)= \tau^{\gamma_X}_{\sigma_{F'X}(h)}) \in\EE.$$ 
		\item Given a 2-function symbol $\sigma\in\LL$ of arity $Y\in\Cat_f$, for any $\eta\colon h\Rightarrow h'\colon Y\to F'X$ in $\Cat$ we consider the $X$-ary 2-term $\sigma_i(\eta)$, where $\sigma_i:=\sigma(\pi_i)$ as in Remark~\ref{power-2-term} and again we see $\eta$ a s a map $\bb 2\times Y\to F'X$. Then
		$$ (\sigma_i(\eta)= \tau^{\gamma_X}_{(\sigma_i)_{F'X}(\eta)}) \in\EE.$$ 
	\end{enumerate}
	We will show that an $\LL$-structure $A$ satisfies (1) and (2) if and only if each $FX\to A$ factors through $\gamma_X$ as a morphism in $\Cat$; the factorization is necessarily unique since $\gamma_X$ is an epimorphism. In addition, $A$ satisfies (3)-(6) if and only if the factorization is a morphism of $\LL$-structures. This is enough since it implies that $A\in\Mod(\EE)$ if and only if it is orthogonal with respect to $\gamma_X$ for any $X\in\Cat_f$, if and only if $A\in\A$. 
	
	Consider an $\LL$-structure $A$ in $\Mod(\EE)$ and a morphism of $\LL$-structures $\hat g\colon FX\to A$. By Theorem~\ref{free-structures-teo}, there exists a unique $g\colon X\to A$ in $\Cat$ such that $\hat g\circ \eta_X= g$; moreover, we have $g([\sigma])=\sigma_A(g)$ for any (2-)term $\sigma$ representing an object or morphism of $FX$. 
	
	Now we define a functor $G\colon F'X\to A$ as follows:\begin{itemize}[leftmargin=0.8cm]
		\item if $z\in F'X$, fix an $X$-ary term $t$ such that $\gamma_X([t])=z$ and define
		$$ G(z):=t_A(g); $$
		\item if $f$ is a morphism in $F'X$, then define
		$$ G(f):= (\tau_f^{\gamma_X})_A(g). $$
	\end{itemize} 
	This is well defined and does not depend on the choice of the representatives since $A$ satisfies the equations in (1) and (2). Moreover, by construction we have that $G\circ\gamma_X=\hat g$ in $\Cat$, and such $G$ is unique because $\gamma_X$ is an epimorphism. Thus we are only left to prove that $G$ is a morphism of $\LL$-structures.
	
	Fix a function symbol $f\in \LL$ of arity $Y$; we need to show that for any $h\colon Y\to F'X$ we have
	\begin{equation}\label{1-L-strc}
		G\circ f_{F'X}(h)=f_A\circ G^Y(h)
	\end{equation}
	and that the same holds for $\eta\colon h\Rightarrow h'\colon Y\to A$ in place of $h$. On the one hand we know that
	$$ G\circ f_{F'X}(h)= (t_h)_A(g), $$
	following the definition of $G$ and the notation of axiom (3). For the other, note that the interpretation of the $X$-ary term $f(h)$ over $A$ is by definition the top composite
	\begin{center}
		\begin{tikzpicture}[baseline=(current  bounding  box.south), scale=2]
			
			\node (a0) at (0,0.8) {$A^X$};
			\node (b0) at (1.2,0.8) {$A^Y$};
			\node (c0) at (1.2,0) {$\prod_i A^{\bb2}$};
			\node (d0) at (2,0.8) {$A$};
			
			\path[font=\scriptsize]
			
			(a0) edge [dashed, ->] node [above] {$(\tau^{\gamma_X}_h)_A$} (b0)
			(a0) edge [->] node [below] {$(\tau^{\gamma_X}_{hy_i})_A\ \ \ \ \ \ $} (c0)
			(b0) edge [>->] node [right] {$A^e$} (c0)
			(b0) edge [->] node [above] {$f_A$} (d0);
		\end{tikzpicture}	
	\end{center}
	where $(\tau^{\gamma_X}_h)_A$ exists (and is unique) because $f(h)$ is interpretable over $A$. Moreover, $(\tau^{\gamma_X}_h)_A(g)= G\circ h$, since by definition $(G\circ h)(y_i)=(\tau^{\gamma_X}_{hy_i})_A(g)$. It follows that
	$$ f_A\circ G^Y(h)= f_A(G\circ h)= f_A(\tau^{\gamma_X}_h)_A(g)=(f(h))_A(g). $$
	Since $A$ satisfies the equation in (3) this implies that condition~\ref{1-L-strc} holds. Arguing in the same way for a morphism $\eta\colon h\Rightarrow h'$ in $A^Y$, it follows from the equation in (4) that the same property holds also for $\eta$ in place of $h$, so that $G$ respects the interpretation of $f$.
	
	Similarly, using that $A$ satisfies (5) and (6) the same arguments can be applied for a 2-function symbol $\sigma\in\LL$. Therefore if $A\in\Mod(\EE)$ then it is orthogonal with respect to $\gamma_X$ for any $X\in\Cat_f$. Conversely, it is clear from the explicit calculation given above that if $A$ is orthogonal with respect to  the $\gamma_X$ then it satisfies all the equations of $\EE$. 
\end{proof}

This can be easily generalised to the case of Section~\ref{souns-section} where a sound class of limits $\Phi$ is considered. Below we take $\Phi$ to be contained in the class of limit limits since we have only treated finitary languages in this section; nonetheless, the result can be proved more generally.

\begin{teo}
	Let $\Phi$ be a locally small and weakly sound class of finite limits. Given a $\Phi$-ary language $\LL$, the following are equivalent for a full subcategory $\A$ of $\Str(\LL)$:\begin{enumerate}[leftmargin=1cm]
		\item $\A\cong\Mod(\EE)$, for some $\Phi$-ary equational theory $\EE$ on $\LL$;
		\item $\A$ is closed un $\Str(\LL)$ under products, powers, strong subobjects, $\V$-split quotients, and $\Phi$-flat colimits.
	\end{enumerate}
\end{teo}
\begin{proof}
	We argue as in the proof of Theorem~\ref{birk-theo}. For $(1)\Rightarrow (2)$, the closure properties are given by Proposition~\ref{closure} and Lemma~\ref{Phi-flat-pres}. 
	
	Assume now that $\A$ is closed un $\Str(\LL)$ under products, powers, strong subobjects, and $\Phi$-flat colimits. Then, the same proof of Theorem~\ref{birk-theo} shows that $\A$ is the (ordinary) orthogonality class defined by the set $\{\gamma_X\}_{X\in\Cat_\Phi}$ (one simply replaces filtered colimits with $\Phi$-flat colimits). Then one argues as in the second part of the proof by taking judgements and equations with arity in $\Cat_\Phi$. 
\end{proof}

\begin{obs}
	When $\Phi=\tx{Fp}$ is the sound class for finite products, we obtain a Birkhoff variety theorem characterizing the 2-category of models of $\tx{Fp}$-ary equational theories; that is, of those equational theories involving only discrete arities. These, by Theorem~\ref{char-single-Phi}, correspond to the 2-categories of algebras of strongly finitary monads on $\Cat$.
	In this context another 2-categorical Birkhoff theorem was proved in~\cite{dostal2016two}; however, this does not seem comparable to ours as we, indirectly, consider the (epi, mono) factorization system on $\Cat$ while in~\cite{dostal2016two} they work on the (b.o.o.+full, faithful) factorization system. 
\end{obs}

%\bibliography{biblio}
%\bibliographystyle{abbrv}

\end{document}